\documentclass[a4paper,10pt]{amsart}
\usepackage[utf8]{inputenc}

\usepackage{amssymb,amsmath,amsthm}
\usepackage{graphicx}
\usepackage{subfigure}
\usepackage{overpic}
\usepackage[table]{xcolor}

\usepackage{chngcntr}
\usepackage{lscape}
\usepackage{rotating}
\usepackage{booktabs}
\usepackage{siunitx}

\usepackage{epstopdf}

\usepackage{multirow}
\usepackage{calc}

\usepackage{tikz}
\usepackage{tikz-cd}
\usepackage{pgfplots}
\usepackage{pgfplotstable}
\usetikzlibrary{3d,calc}
\usetikzlibrary{shapes, spy, scopes, matrix, datavisualization.formats.functions, arrows, datavisualization, plothandlers, plotmarks, svg.path, positioning}
\pgfplotsset{compat=newest}
\pgfplotsset{plot coordinates/math parser=false}
\pgfplotsset{scaled y ticks = false, tick label style={/pgf/number format/fixed}}
\pgfplotsset{scaled x ticks = false, tick label style={/pgf/number format/fixed}}
\pgfplotsset{compat=1.6}

\pgfplotsset{colormap/viridis reversed/.style={/pgfplots/colormap={viridisrev}{
  rgb255=(253,231,37) rgb255=(94,201,98) rgb255=(33,145,140)
  rgb255=(59,82,139) rgb255=(68,1,84)}}}

\usepackage{float}

\usepackage[bookmarksopen, pdfstartview={FitV}, linkbordercolor=white]{hyperref}

\usepackage{algorithm}
\usepackage{algpseudocode}
\usepackage{rotating}

\usepackage[capitalize,nameinlink]{cleveref}
	
\newtheorem{theorem}{Theorem}[section]
\newtheorem{definition}[theorem]{Definition}
\newtheorem{lemma}[theorem]{Lemma}
\newtheorem{remark}[theorem]{Remark}
\newtheorem{corollary}[theorem]{Corollary}
\newtheorem{proposition}[theorem]{Proposition}

\newcommand{\BIGOP}[1]{\mathop{\mathchoice%
{\raise-0.22em\hbox{\huge $#1$}}%
{\raise-0.05em\hbox{\Large $#1$}}{\hbox{\large $#1$}}{#1}}}

\newcommand{\BIGboxplus}{\mathop{\mathchoice%
{\raise-0.35em\hbox{\huge $\boxplus$}}%
{\raise-0.15em\hbox{\Large $\boxplus$}}{\hbox{\large $\boxplus$}}{\boxplus}}}

\def\N{\mathbb N}

\def\R{\mathbb R}

\def\R{\mathbb{R}}
\def\N{\mathbb{N}}

\def\cA{\mathcal{A}}

\def\cF{\mathcal{F}}

\def\cH{\mathcal{H}}
\def\cI{\mathcal{I}}
\def\cJ{\mathcal{J}}

\def\cO{\mathcal{O}}

\def\cV{\mathcal{V}}
\def\cW{\mathcal{W}}

\def\aa{\boldsymbol{a}}

\def\cc{\boldsymbol{c}}

\def\ee{\boldsymbol{e}}
\def\ff{\boldsymbol{f}}

\def\uu{\boldsymbol{u}}

\def\xx{\boldsymbol{x}}
\def\yy{\boldsymbol{y}}

\def\11{\mathbf{1}}

\def\00{\boldsymbol{0}}

\def\aalpha{\boldsymbol{\alpha}}
\def\bbeta{\boldsymbol{\beta}}

\def\iiota{\boldsymbol{\iota}}
\def\kkappa{\boldsymbol{\kappa}}
\def\llambda{\boldsymbol{\lambda}}

\def\nnu{\boldsymbol{\nu}}
\def\xxi{\boldsymbol{\xi}}

\def\ssigma{\boldsymbol{\sigma}}
\def\ttau{\boldsymbol{\tau}}

\def\oomega{\boldsymbol{\omega}}

\def\OOmega{\boldsymbol{\Omega}}

\def\transpose{\text{T}}

\def\spn{\operatorname{span}}

\title[Anisotropic Kernel-based Multilevel Interpolation]{Anisotropic Kernel-based Multilevel Interpolation of High-Dimensional Functions on Sparse Grids}

\author{R\"udiger Kempf}

\begin{document}

\begin{abstract}
The \emph{tensor product multilevel method (TPML)} approximates high-dimensional functions from scattered data on general bounded domains by combining Smolyak's sparse grid construction with the kernel-based multilevel interpolation, coupling the kernel scale to the fill distance at each level so that condition numbers stay uniformly bounded. Its original formulation depended only implicitly on the target function, which precluded numerical evaluation. Using a recent nodal representation of the low-dimensional multilevel operators, we derive an explicit, computable representation of the high-dimensional interpolant. On this basis we prove error estimates in mixed-regularity Sobolev and continuity norms, extending the existing $ L_2 $-theory, and we characterise the full range of anisotropy weights that attain the optimal rate against the number of degrees of freedom. Three explicit weight strategies follow, equilibrating the accuracy, the degrees of freedom, or the cost-benefit ratio. Numerical experiments in up to ten dimensions illustrate the theory.
\end{abstract}

%
%

\maketitle

\section{Introduction}
\label{sec:Introduction}
High-dimensional problems emerge in a wide range of applications from physics, engineering and even financial science. Typical examples comprise solutions of the Schr\"odinger equations, the Hamilton-Jacobi-Bellman equation, the Fokker-Planck equation or the Black-Scholes equation. Another application is parametric partial differential equations in the context of uncertainty quantification, where the solution or more generally the quantity of interest is then a high-dimensional function of the parameters.

One limiting factor in treating these high-dimensional problems is the curse of dimensionality \cite{Bellman1957}, that is, to achieve an approximation error of $ \varepsilon $ one needs the number of data to grow exponentially in the dimension. This is already for a moderately high dimension not feasible. An efficient tool to alleviate this curse are sparse grid techniques, see for example \cite{Bungartz2004} for an overview. The underlying procedure, originally developed by Smolyak \cite{Smolyak1963} for quadrature, is frequently used and studied, see \cite{Barthelmann2000,Conrad2013,Ullrich2008,Nobile2008} among many others. An equivalent representation of the approximation operator, in case of interpolation, is the combination technique \cite{Bungartz1994,Garcke2007,Garcke2009,Griebel1990,Hegland2007,Hegland2016}, allowing easier implementation and computation.

The basic idea of Smolyak's construction is to combine sequences of approximation operators on low-dimensional domains in an orderly fashion to obtain a high-dimensional approximation method which exhibits nearly the same approximation features as the building-block operators with an alleviated curse of dimensionality.

Most often, these sparse grid techniques are limited to the Cartesian product of one dimensional intervals and approximation schemes are either spline or polynomial interpolation.
In contrast, the use of kernel methods in the context of sparse grids have become more and more popular in recent years, see, e.g., \cite{Buettner2025,griebel2025:2,Griebel2025,Kempf2023}. These methods have the advantage that one can use more general low-dimensional domains and scattered data which may be important if the solution of a space-time partial differential equation is studied. We refer to \cite{Wendland2004} for a general introduction to kernel methods.

While \cite{griebel2025:2,Griebel2025} consider standard single-level kernel interpolation, it does not leverage the intrinsic low-dimensional multiscale structure of sparse grids. They employ kernels with a single, fixed correlation length. This means that their kernel matrices become increasingly ill-conditioned as the number of points grows. In \cite{Kempf2023}, the present author proposed the \emph{tensor product multilevel method (TPML)}. It combines the Smolyak approach with the kernel-based multilevel method, a residual correction scheme that uses the sequence of low-dimensional data point sets. This method is thoroughly studied, see, e.g., \cite{Avesani2025,LeGia2010,Wendland2010,Wendland2017} and provides a high-order and stable interpolation process for low-dimensional domains using scattered data. The basic idea is, that on each level, we approximate the residual of the previous level with a level-dependently scaled kernel. This allows us to control the condition numbers of the resulting linear systems. The global approximant is then taken as the sum of all those detail approximations. The downside, however, is that the standard formulation only contains an implicit dependence of the data-generating function which makes tensorization impossible. First steps towards an explicit dependence were taken in \cite{Kempf2023} and \cite{Buettner2025}. Finally, \cite{Gollwitzer2026} derived the formula for the multilevel operator we leverage here to obtain a new representation of the TPML interpolation operator for high-dimensional functions. This representation allows a numerical feasible computation of the interpolant, improving on the results of \cite{Buettner2025,Kempf2023}. Additionally, we build upon results from \cite{Griebel2012,Griebel2013} and derive a new error analysis for TPML in mixed regularity Sobolev and continuity norms. We are also able to discuss the optimal choice of weights that determine the anisotropy in the sparse grid approach, extending the discussion of \cite{Kempf2023}. Because the representation of the high-dimensional interpolation operator is now numerically computable, we are able to test the method in three settings. First, we demonstrate the derived bounds on the numerical costs in a standard, isotropic setting in up to twelve dimensions. Second, we test the three choices for the weights in a $1+2+3$-dimensional example supporting the derived error bounds. In this example, we particularly use a tensor product function of finite smoothness. Finally, in the third example, we demonstrate that the method is independent of the structure of the target function by approximating a ten dimensional, non-tensor product Genz test function.

The rest of this paper is structured as follows: In \Cref{sec:Preliminaries} we repeat the necessary building blocks for the TPML. We introduce the kernel-based multilevel method and quote the needed representation of the interpolation operator. Furthermore, we give the results on tensor products of spaces and operators we use in the error analysis later. We also give a short overview over Smolyak's method. In \Cref{sec:ConvergenceAnalysisAndOptimalWeights} we first derive the new representation of the TPML operator and derive the new convergence results. Additionally, we discuss the selection of optimal weights. We give a short overview over the implementation and give the algorithms to compute the TPML interpolant in \Cref{sec:Implementation}. Finally, we give extensive numerical tests in \Cref{sec:NumericalExperiments}. To end, we give a conclusion and outlook in \Cref{sec:Conclusion}.

\section{Preliminaries}\label{sec:Preliminaries}
We start with repeating the necessary building blocks for the TPML: The kernel-based multilevel method, some basic facts about tensor products of spaces and operators and Smolyak's construction of a high-dimensional approximation operator.

\subsection{Kernel-based Multilevel Interpolation}
Let $ \Omega \subseteq \R^n $, $ n \in \N $, be a domain and $ X = \{\xx_1, \dots, \xx_N \} \subseteq \Omega $ a discrete set of sites. We work in a \emph{reproducing kernel Hilbert space (RKHS)} $ \cH $ with unique reproducing kernel $ K: \Omega \times \Omega \to \R $. We are particularly interested in \emph{translation invariant} and \emph{radial kernels}, i.e., $ K(\xx,\yy) = \Phi(\xx - \yy) = \phi(\|\xx - \yy \|_2) $ for functions $ \Phi: \R^n \to \R $ and $ \phi: [0, \infty) \to \R $. Moreover, if $ \sigma > n/2 $ and $ \Phi \in L_1(\R^n) \cap C(\R^n) $ is such that its Fourier transform satisfies
\begin{align}\label{eq:FourierDecay}
 c_1 \left(1+ \| \xxi \|_2^2 \right)^{-\sigma} \leq \widehat{\Phi}(\xxi) \leq c_2 \left( 1 + \| \xxi \|_2^2 \right)^{-\sigma}, \quad \xxi \in \R^n,
\end{align}
with constants $ c_1, c_2 > 0 $, then the translation invariant kernel $ \Phi $ is the reproducing kernel of $ H^{\sigma}(\R^n) $ \cite{Wendland2004}. Typical examples, where \eqref{eq:FourierDecay} holds, are Mat\'ern \cite{Matern1986} and Wendland's compactly supported kernels \cite{Wendland1995,Wendland2004}. 

In applications, we scale the radial kernel $ \Phi $ by a \emph{scaling parameter} $ \delta > 0 $, and define
\begin{align}\label{eq:DefRescaledRBF}
 \Phi_{\delta} := \delta^{-n} \Phi(\cdot / \delta).
\end{align}
The parameter $ \delta $ controls the correlation length of kernel, and therefore, the localization of the resulting kernel basis functions. For a given set of sites $ X $, we consider the space
\begin{align*}
 W := \spn\{\Phi_{\delta}(\cdot - \xx_i) \; : \; \xx_i \in X \}.
\end{align*}
Then corresponding kernel interpolant is the unique function in $ W $ that reproduces prescribed function values at the sites.

\begin{definition}\label{def:Interpolation}
 Let $ X = \{ \xx_1, \dots, \xx_N \} \subseteq \Omega \subseteq \R^n $ be a set of sites. For $ f \in C(\Omega) $, the kernel interpolant with respect to $ X $ and $ \Phi_{\delta} $ is defined as
 \begin{align*}
  I_{X,\Phi_{\delta}}(f) = \sum_{i=1}^{N} c_i \Phi_{\delta}(\cdot - \xx_i), \quad f \in C(\Omega),
 \end{align*}
 where the coefficient vector $ \cc \in \R^N $ is determined by the interpolation conditions $ I_{X,\Phi_{\delta}}(f)(\xx_i) = f(\xx_i) $, $ 1 \leq i \leq N $, i.e., $ \cc $ is the solution of the linear system 
 \begin{align}\label{eq:InterpolationProblem}
  M_{X,\Phi_{\delta}} \cc = \ff,
 \end{align}
with $ M_{X,\Phi_{\delta}} = \left( \Phi_{\delta}(\xx_i - \xx_k) \right) $ and $ \ff = (f(\xx_i)) $.
\end{definition}

For the classes of Mat\'ern and Wendland kernels the system in \eqref{eq:InterpolationProblem} is uniquely solvable.

The choice of the scaling parameter $ \delta $ involves a fundamental trade-off. If $ \delta $ is kept fixed while the number of sites increases, the kernel basis functions become increasingly redundant and the condition number of $ M_{X,\Phi_\delta} $ typically deteriorates. On the other hand, choosing $ \delta $ proportional to the fill distance in order to maintain a favorable condition number can prevent convergence as the discretization is refined. This is commonly referred to as the \emph{trade-off principle} \cite{Wendland2017}. The multilevel method resolves this tension.

The idea is to apply a residual correction scheme to a hierarchy of kernel approximations \cite{Floater1996}: Take a sequence of data point sets $ X_1, X_2, \dots \subseteq \Omega $ with \( X_{\ell} := \{ \xx_{\ell,1}, \dots, \xx_{\ell, N_{\ell}} \} \), radial kernels $ \Phi_1, \Phi_2, \dots $ and set $ W_{\ell} := \spn\{ \Phi_{\ell}( \cdot - \xx_k) \; : \; \xx_k \in X_{\ell} \}$. The method then sets $ f_0 = 0 $, $ e_0 = f $, and for $ \ell = 1, 2, \dots $ computes a local approximant $ s_\ell \in W_\ell $ to the residual of the previous level $ e_{\ell-1} $, and finally updates $ f_\ell = f_{\ell-1} + s_\ell $ and $ e_\ell = e_{\ell-1} - s_\ell $.

The choice of $ \Phi_{\ell} $ is usually coupled to $ X_{\ell} $. To this end we define the \emph{fill distance} of $ X_{\ell} $ 
\begin{align*}
 h_{\ell} := h_{X_{\ell}, \Omega} := \sup_{\xx \in \Omega} \min_{1 \leq i \leq N_{\ell}} \| \xx - \xx_{\ell,i} \|_2.
\end{align*}
For results later, we need the sets $ X_{\ell} $ to be quasi-uniform: Define the separation radius $ q_{\ell} $ of $ X_{\ell} $ as 
\begin{align*}
 q_{\ell} := q_{X_{\ell}} := \frac{1}{2} \min_{1 \leq i \neq j \leq N_{\ell}} \| \xx_i - \xx_j \|_2.
\end{align*}
The sequence of point sites $ (X_{\ell}) $ is \emph{quasi-uniform} if there is a constant $ c_{qu} > 0 $ such that the \emph{mesh-ratio} $ \rho_{\ell} := h_{\ell}/q_{\ell} $ satisfies $ \rho_{\ell} < c_{qu} $, for all $ \ell \in \N $.

Although the sets $ X_1,X_2, \dots$ do not need to be nested, the fill distances need to decay uniformly with a \emph{refinement parameter} $ \mu \in (0,1) $, so that $ h_{\ell+1} \approx \mu h_\ell $ for all $ \ell $. Most often, we fix one radial kernel $ \Phi $ and scale it as in \eqref{eq:DefRescaledRBF} with $ \delta_\ell = \nu h_\ell $ for a fixed \emph{overlap} $ \nu > 1 $, giving $ \Phi_\ell := \Phi_{\delta_\ell} $. This setting allows us to define an approximation operator \cite{Wendland2010}.

\begin{theorem}\label{cor:MultilevelOperator}
 If the local approximations $ s_{\ell} \in W_{\ell} := \spn\{ \Phi_{\ell}(\cdot - \xx_{\ell,i}) \; : \; \xx_{\ell,i} \in X_{\ell} \} $ are computed with the interpolation operator of \cref{def:Interpolation}, then for each $ L \in \N $ there is a linear and bounded \emph{multilevel interpolation operator} $ A_L: C(\Omega) \to V_L := W_1 + \cdots + W_L $ given by
 \begin{align*}
  A_L(f) = \sum_{\ell = 1}^{L} s_{\ell} = \sum_{\ell =1}^{L} I_{X_{\ell}, \Phi_{\ell}}(e_{\ell-1}).
 \end{align*}
\end{theorem}

The convergence of the multilevel method was proven in \cite{Wendland2010} for the $ L_2 $-norm, in \cite{Wendland2017} for $ L_{\infty} $, and in the following general form in \cite{Gollwitzer2026}.

\begin{theorem}\label{thrm:ConvergenceMultilevel}
Let $ \Omega \subseteq \R^n $ be a bounded Lipschitz domain. Let $ X_1, X_2, \dots $ be a sequence of point sets in $ \Omega $ with fill distances satisfying $ c \mu h_{\ell} \leq h_{\ell +1} \leq \mu h_{\ell} $ for $ \ell = 1, 2, \dots $ with fixed $ \mu \in (0,1) $, $ c \in (0,1] $ and $ h_1 $ sufficiently small. Let $ \sigma > n/2 $ and $ \Phi$ be the reproducing kernel of $ H^{\sigma}(\R^n) $, i.e., its Fourier transform satisfies \eqref{eq:FourierDecay}, and let $ \Phi_\ell $ be defined by \eqref{eq:DefRescaledRBF} with $ \delta_{\ell} = \nu h_{\ell} $. Assume $ 1/h_1 \geq \nu \geq \gamma/\mu $ with a fixed $\gamma > 0 $. Let $ q \in [1,\infty]$. Then there exist constants $ C, C_1 > 0 $ such that for all $ f \in H^{\sigma}(\Omega) $, with $\widetilde{\alpha} = C_1 \mu^{\sigma - \tau- n (\frac{1}{2} - \frac{1}{q})_+}$,
\begin{align}\label{eq:ErrorEstimate}
  \| f - A_L(f) \|_{W^{\tau}_q(\Omega)} \leq C \widetilde{\alpha}^{L} \| f \|_{H^{\sigma}(\Omega)},
\end{align}
holds for all $ 0 \leq \tau \leq \widetilde{\sigma} $, where $ \widetilde{\sigma} = \widetilde{\sigma}_0 := \sigma - n(1/2 - 1/q)_+ $ in the case $ \sigma \in \N $ and either $ q > 2 $ and $ \widetilde{\sigma}_0 \in \N $, or $ q = 2 $; otherwise $ \widetilde{\sigma} = \lceil \widetilde{\sigma}_0 \rceil - 1 $. Moreover, $ \tau \in \N_0 $ if $ q = \infty $.
\end{theorem}

Since $C_1$ is unknown in practice, the following corollary, obtained by accepting a slight loss in convergence order \cite[Section 4.2]{Wendland2017}, is more useful.

\begin{corollary}\label{cor:ConvergenceMultilevel}
 With the notation and assumptions of \cref{thrm:ConvergenceMultilevel}, if there is $ \varepsilon > 0 $ such that $ \sigma - \tau - n(\frac{1}{2} - \frac{1}{q})_+ > \varepsilon $ and $ C_1 \mu^{\varepsilon} \leq 1 $, then
 \begin{align*}
  \| f - A_L(f) \|_{W^{\tau}_q(\Omega)} \leq C \alpha^L \| f \|_{H^{\sigma}(\Omega)},
 \end{align*}
 where $ \alpha = \mu^{\sigma - \tau - \varepsilon - n (\frac{1}{2} - \frac{1}{q})_+} $.
\end{corollary}

Convergence holds whenever $ \mu $ is chosen such that $ \alpha < 1 $, which we assume throughout. Moreover, the result in \cref{cor:ConvergenceMultilevel} implies that $ A_L $ is bounded. 

\begin{remark}
It is also possible to derive convergence estimates, if the target function $ f $ is in $ H^{\varsigma}(\Omega) $ with $ n/2 < \varsigma < \sigma $, i.e., it is not as smooth as the kernel used in the approximation process. This is commonly known as \emph{escaping the native space} and we refer to \cite[Section 5]{Wendland2010} for details.
\end{remark}

In \cite{Gollwitzer2026}, an alternative nodal representation of $ A_L $ was observed, relying on the \emph{Lagrange functions} $ \chi_{L,i} \in W_L $ defined by $ \chi_{L,i}(\xx_{L,k}) = \delta_{ik} $. We restrict ourselves to nested point sets.

\begin{theorem}\label{thrm:NodalRepresentationMultilevelOperator}
Assume $ X_1 \subseteq X_2 \subseteq \cdots \subseteq X_L $. Then the multilevel operator $ A_L $ admits the representation
 \begin{align}\label{eq:NodalRepresentationMultilevelOperator}
  A_L(f) = \sum_{i = 1}^{N_L} f(\xx_{L,i}) b_{L,i}, \quad f \in C(\Omega),
 \end{align}
 where $ b_{L,i} = A_L(\chi_{L,i}) $, $ 1 \leq i \leq N_L $.
\end{theorem}
\begin{proof}
Since $ A_L $ only uses the data $ f|_{X_L} $, we have $ A_L(f) = A_L(I_{X_L,\Phi_L}(f)) $. Inserting $ I_{X_L,\Phi_L}(f) = \sum_{i=1}^{N_L} f(\xx_{L,i}) \chi_{L,i} $ and using linearity of $ A_L $ gives the result.
\end{proof}

The representation in \cref{thrm:NodalRepresentationMultilevelOperator} is remarkable since the dependence on the data $ f(\xx_{L,i}) $ is explicit.

Furthermore, we need the following properties of the image space of $ A_L$, also taken from \cite{Gollwitzer2026}.

\begin{corollary}\label{cor:PropertiesImageSpace}
 Assume $ X_1 \subseteq X_2 \subseteq \cdots \subseteq X_L $. Then
\begin{enumerate}
 \item $ \operatorname{dim}(A_L(C(\Omega))) = N_L $;
 \item $ A_1(C(\Omega)) \subseteq A_2(C(\Omega)) \subseteq \cdots \subseteq A_L(C(\Omega)) $.
\end{enumerate}
\end{corollary}

\subsection{Tensor Products of Spaces and Operators}
We give a brief introduction to tensor products of function spaces and operators; for details and proofs we refer to, e.g., \cite{Aubin2011,Hackbusch2012,Light2006}.

\begin{definition}
 For $ 1 \leq j \leq d $ let $ U^{(j)} $ be linear spaces of functions $ u^{(j)}: \Omega^{(j)} \to \R $, where $ \Omega^{(j)} \subseteq \R^{n_j} $. The \emph{elementary tensor} $ \uu := u^{(1)} \otimes \cdots \otimes u^{(d)} $ is defined by
 \begin{align*}
  \uu(\xx) := \prod_{j=1}^{d} u^{(j)}(\xx^{(j)}), \quad \xx = ( \xx^{(1)}, \dots, \xx^{(d)}) \in \OOmega := \Omega^{(1)} \times \cdots \times \Omega^{(d)},
 \end{align*}
 and the \emph{algebraic tensor product space} is $ U^{(1)} \otimes \cdots \otimes U^{(d)} := \spn\{ u^{(1)} \otimes \cdots \otimes u^{(d)} \; : \; u^{(j)} \in U^{(j)}, \ 1 \leq j \leq d \} $.
\end{definition}

We are also interested in the tensor product of operators. 

\begin{definition}
 Let $ U^{(j)}, V^{(j)} $, $ 1 \leq j \leq d $, be linear spaces with algebraic tensor products $ S = U^{(1)} \otimes \cdots \otimes U^{(d)} $ and $ T = V^{(1)} \otimes \cdots \otimes V^{(d)} $, and let $ A^{(j)}: U^{(j)} \to V^{(j)} $ be linear operators. The \emph{tensor product operator} $ \cA := A^{(1)} \otimes \cdots \otimes A^{(d)} : S \to T $ is defined by
 \begin{align*}
  \cA\!\left(\sum_{i=1}^{m} u^{(1)}_i \otimes \cdots \otimes u^{(d)}_i\right) := \sum_{i=1}^{m} A^{(1)}u^{(1)}_i \otimes \cdots \otimes A^{(d)} u^{(d)}_i, \quad m \in \N.
 \end{align*}
\end{definition}

For the error analysis later, we need tensor products of $ L_p $-spaces, Sobolev spaces $W^{\tau}_p $ and spaces of continuously differentiable functions. Since the tensor product is commutative and associative up to isomorphism \cite{Hackbusch2012}, we restrict ourselves to $ d = 2 $ for the remainder of this subsection.

\begin{definition}\label{def:Wmix}
Let $ \Omega^{(1)} \subseteq \R^{n_1} $, $ \Omega^{(2)} \subseteq \R^{n_2} $ be measurable, and $ \OOmega = \Omega^{(1)} \times \Omega^{(2)} $. For $ \ttau \in \N_0^2 $ and $ 1 \leq p < \infty $, the \emph{Sobolev space of mixed regularity} is
 \begin{align*}
  W^{\ttau,p}_{mix}(\OOmega) := \{ \uu \in L_p(\OOmega) \; : \; D^{\aalpha_1,\aalpha_2} \uu \in L_p (\OOmega), \ |\aalpha_j | \leq \tau_j, \ j = 1,2 \},
 \end{align*}
with norm $ \| \uu \|_{W^{\ttau,p}_{mix}} := \left( \sum_{| \aalpha_1| \leq \tau_1} \sum_{| \aalpha_2 | \leq \tau_2} \| D^{\aalpha_1,\aalpha_2} \uu \|^p_{L_p(\OOmega)} \right)^{1/p} $.
\end{definition}

This space extends to non-integer smoothness \cite{Hansen2010,Sickel2009,Sickel2011} and can be identified as a tensor product space.

\begin{theorem}\label{thrm:TensorProductSobolevSpaces}
 In the setting of \cref{def:Wmix}, for every $ 1 \leq p < \infty $, there is a norm such that the completion of $ W^{\tau_1,p}(\Omega^{(1)}) \otimes W^{\tau_2,p}(\Omega^{(2)}) $ with respect to this norm is isomorphic to $ W^{\ttau,p}_{mix}(\OOmega) $.
\end{theorem}

For $ \ttau = \00 $, the norm in \cref{thrm:TensorProductSobolevSpaces} is the \emph{p-nuclear norm} $ \alpha_p $, if $ 1 < p < \infty $, and the \emph{injective norm} $ \gamma $, if $ p=1 $, cf. \cite{Light2006}. If $ p=2 $, we adopt the notation $ W^{\ttau,2}_{mix}(\OOmega) =: H^{\ttau}_{mix}(\OOmega) $, which can be shown to be a Hilbert space. 

Similar to \cref{def:Wmix} we define the space of continuously differentiable functions of mixed regularity.

\begin{definition}\label{def:Cmix}
In the setting of \cref{def:Wmix}, for $ \ttau \in \N_0^2 $, the space of \emph{continuously differentiable and bounded functions of mixed regularity} $ C^{\ttau}_{b,mix}(\OOmega)$
\begin{align*}
 C^{\ttau}_{b,mix}(\OOmega) := \{ \uu \in C(\OOmega) \; : \; D^{\aalpha_1,\aalpha_2} \uu \in C(\OOmega), \ |\aalpha_j | \leq \tau_j, \ j = 1,2 \},
\end{align*}
with norm $ \| \uu \|_{C^{\ttau}_{b,mix}} := \max_{| \aalpha_1 | \leq \tau_1} \max_{| \aalpha_2 | \leq \tau_2} \| D^{\aalpha_1,\aalpha_2} \uu \|_{L_{\infty}(\OOmega)} $.
\end{definition}

Again, we can interpret $ C^{\ttau}_{b,mix}(\OOmega) $ as a tensor product space \cite{Hackbusch2012}. 

\begin{theorem}\label{thrm:TensorProductContinuousSpaces}
 In the setting of \cref{def:Cmix} there is a norm such that the completion of the tensor product space $ C^{\tau_1}_{b}(\Omega^{(1)}) \otimes C^{\tau_2}_b(\Omega^{(2)}) $ with respect to this norm is isomorphic to $ C^{\ttau}_{b,mix}(\OOmega) $.
\end{theorem}

For $ \ttau = \00 $, the norm in \cref{thrm:TensorProductContinuousSpaces} is called \emph{projective norm} $ \lambda $, cf. \cite{Light2006}.

The following norm equality, see \cite{Hansen2010}, is essential for the error estimates in \Cref{subsec:ErrorEstimates}.

\begin{theorem}\label{thrm:TensorProductNormUniform}
For $ j = 1,2 $ let $ U^{(j)} $, $ V^{(j)} $ be both either $ W^{\tau,p} $- or $ C^{\tau}_b $-spaces and $ A^{(j)}: U^{(j)} \to V^{(j)} $ be linear and bounded operators. Then
 \begin{align*}
  \| A^{(1)} \otimes A^{(2)} \|_{U^{(1)} \otimes U^{(2)} \to V^{(1)} \otimes V^{(2)}} = \| A^{(1)} \|_{U^{(1)} \to V^{(1)}} \cdot \| A^{(2)} \|_{U^{(2)} \to V^{(2)}}.
 \end{align*}
\end{theorem}

An immediate example of a tensor product operator is the Sobolev embedding operator.

\begin{corollary}\label{cor:TensorProductEmbeddingOperator}
 In the setting of \cref{def:Wmix}, let $ \sigma_j > \tau_j + n_j \left( \frac{1}{2} - \frac{1}{q} \right)_+ $ for $ j = 1,2 $, and denote by $ \iota^{(j)}: H^{\sigma_j}(\Omega^{(j)}) \to W^{\tau_j,q}(\Omega^{(j)}) $ the Sobolev embedding operator. Then the tensor product operator $ \iiota := \iota^{(1)} \otimes \iota^{(2)} $ defines a bounded embedding 
 \begin{align*}
 H^{\ssigma}_{mix}(\OOmega) \hookrightarrow W^{\ttau,q}_{mix}(\OOmega).
 \end{align*}
 Analogously, if $ \sigma_j > \tau_j + n_j/2 $ and $ \tau_j \in \N_0 $, then $ H^{\ssigma}_{mix}(\OOmega) \hookrightarrow C^{\ttau}_{b,mix}(\OOmega) $.

\end{corollary}

\subsection{Smolyak and Sparse Grid Approximation}\label{subsec:SmolyakApproximation}

The Smolyak construction mitigates the curse of dimensionality by combining a hierarchy of univariate approximation operators into a multivariate scheme that achieves accuracy comparable to full tensor product methods at a fraction of the cost.

\begin{definition}\label{def:AnisotropicIndexSet}
 Let $ \oomega \in \R^d_+ $ be such that $ \omega_{min} := \min_{1 \leq j \leq d} \omega_j = 1 $ and $ T \in \N $. Define the \emph{anisotropic index set} $ \cI_T^{\oomega} $ by 
 \begin{align*}
  \cI^{\oomega}_T = \{ \llambda \in \N_0^d \; : \; \llambda^{\transpose} \oomega \leq T \},
 \end{align*}
and the \emph{surface} of $ \cI^{\oomega}_T $ is given as 
\begin{align*} 
\cJ^{\oomega}_T := \cI^{\oomega}_T \setminus \cI^{\oomega}_{T - \| \oomega \|_1}.
\end{align*}
\end{definition}

The largest possible entry of an element $ \llambda \in \cI^{\oomega}_T $ in direction $ j $ can be directly computed to be 
\begin{align}\label{eq:Lambda_j,max}
 \lambda_{j,max} := \left\lfloor \frac{T}{\omega_j} \right\rfloor.
\end{align}

The weights $ \omega_j $ encode the relative importance of each coordinate direction and govern the trade-off between approximation error and computational work.

\begin{definition}\label{def:SmolyakOperator}
 For $ 1 \leq j \leq d $ let $ \Omega^{(j)} \subseteq \R^{n_j} $ be a domain and $ A^{(j)}_i: C(\Omega^{(j)}) \to V^{(j)}_i $ be a sequence of low-dimensional approximation operators for $ i \in \N $. Set $ A^{(j)}_{-1} = 0 $. Define the \emph{difference operators} $ \Delta^{(j)}_i := A_i^{(j)} - A^{(j)}_{i-1} $ and the \emph{detail spaces} $ W^{(j)}_i = \Delta^{(j)}_i (C(\Omega^{(j)})) $.
 The \emph{sparse grid space} $ \cV^{\oomega}_T $ is given by 
 \begin{align*}
  \cV^{\oomega}_T := \bigoplus_{\llambda \in \cI^{\oomega}_T} \cW_{\llambda}, \quad \text{where } \cW_{\llambda} := \bigotimes_{j=1}^{d} W^{(j)}_{\lambda_j}.
 \end{align*}
 The \emph{Smolyak operator} $ \cA^{\oomega}_T: C(\OOmega) \to \cV^{\oomega}_T $ is then defined by 
 \begin{align*}
  \cA^{\oomega}_T := \sum_{\llambda \in \cI^{\oomega}_T} \bigotimes_{j=1}^{d} \Delta^{(j)}_{\lambda_j} = \sum_{\llambda \in \cI^{\oomega}_T} \bigotimes_{j=1}^{d} \left( A^{(j)}_{\lambda_j} - A^{(j)}_{\lambda_j -1} \right).
 \end{align*}
\end{definition}

For functions of mixed regularity, this construction achieves nearly optimal convergence rates \cite{Bartel2023,Bungartz2004,Smolyak1963}. If $ A_i^{(j)}$ are interpolation operators, the Smolyak operator also admits an equivalent representation via the \emph{anisotropic combination technique}. 

\begin{corollary}\label{cor:CombinationTechniqueRepresentation}
Let $ A^{(j)}_i $ be interpolation operators, then the Smolyak operator $ \cA^{\oomega}_T$ from \cref{def:SmolyakOperator} can be represented as
 \begin{align*}
 \cA^{\oomega}_T = \sum_{\llambda \in \cJ^{\oomega}_T} \sum_{\substack{\bbeta \in \{0,1\}^d \\ \llambda + \bbeta \in \cI^{\oomega}_T}} (-1)^{\| \bbeta \|_1} \bigotimes_{j=1}^{d} A^{(j)}_{\lambda_j}.
\end{align*}
\end{corollary}

Following \cite{Griebel2012,Griebel2013}, three choices of $ \oomega $ arise naturally: equilibrating the \emph{convergence rates} of the extremal spaces $ V^{(j)}_{\lambda_{j,max}} $ across directions, equilibrating their \emph{dimensions} $ \operatorname{dim}(V^{(j)}_{\lambda_{j,max}}) $, or equilibrating the \emph{cost-benefit ratio}, i.e., equilibrating $ \operatorname{dim}(V^{(j)}_{\lambda_{j,max}}) / \text{rate}(A^{(j)}_{\lambda_{j,max}}) $ across directions so that the product equals a fixed base value less than $ 1 $, raised to the power $ T $.

\section{Convergence Analysis and Optimal Weights for TPML}\label{sec:ConvergenceAnalysisAndOptimalWeights}

We are now in the position to derive an improved representation of the tensor product multilevel operator and give a more general convergence analysis. These are mainly a consequence of \cref{thrm:NodalRepresentationMultilevelOperator} and \cref{cor:PropertiesImageSpace}. Moreover, we discuss optimal choices of weights used in the anisotropic combination technique.

\subsection{Representation of the Approximation Operator}\label{subsec:RepresentationTPML}

To combine the kernel-based multilevel method and the ideas of the Smolyak method, we need the setting of \cref{cor:MultilevelOperator} for every direction: For $ 1 \leq j \leq d $, let $ \Omega^{(j)} \subseteq \R^{n_j} $ be a bounded domain and let $ \OOmega = \Omega^{(1)} \times \cdots \times \Omega^{(d)} $. Fix for $ 1 \leq j \leq d $ a sequence of sites 
\begin{align*}
 X_{\ell}^{(j)} = \{ \xx_{\ell,1}^{(j)}, \dots, \xx^{(j)}_{\ell, N^{(j)}_{\ell}} \} \subseteq \Omega^{(j)} \subseteq \R^{n_j}, \quad 1 \leq \ell \leq L^{(j)},
\end{align*}
where $ L^{(j)} = \max \{\lambda_j  : \llambda \in \cI_T^{\oomega} \}$. The corresponding fill distances and scaling parameters are coupled according to 
\begin{align*}
 c^{(j)} \mu_j h^{(j)}_{\ell} &\leq h^{(j)}_{\ell +1} \leq \mu_j h_{\ell}^{(j)}, &\quad 1 \leq \ell \leq L^{(j)} - 1, \ 1 \leq j \leq d, \\ 
 \delta_{\ell}^{(j)} &= \nu_j h^{(j)}_{\ell}, &\quad 1 \leq \ell \leq L^{(j)}, \ 1 \leq j \leq d, 
\end{align*}
with fixed $ c^{(j)} \in (0,1] $, $ \mu_j \in (0,1) $ and $ \nu_j > 1 $.

Next, for each $ 1 \leq j \leq d $, fix a reproducing, radial kernel $ \Phi^{(j)} $ of $ H^{\sigma_j}(\R^{n_j}) $, $ \sigma_j > n_j /2 $, and obtain rescaled kernels $ \Phi^{(j)}_{\ell} $ as in \eqref{eq:DefRescaledRBF} using $ \delta^{(j)}_{\ell} $. This yields the direction-dependent approximation spaces 
\begin{align*}
 W^{(j)}_{\ell} = \spn\{\Phi^{(j)}_{\ell}(\cdot- \xx^{(j)}_{\ell,i}) : \xx^{(j)}_{\ell,i} \in X^{(j)}_{\ell} \}, \quad V^{(j)}_{\ell} = W^{(j)}_1 + \cdots + W^{(j)}_{\ell}.
\end{align*}
Finally, for $ 1 \leq j \leq d $ and $ 1 \leq \ell \leq L^{(j)} $, we have the interpolation operator
\begin{align*}
 A^{(j)}_{\ell}(f) = \sum_{i=1}^{N^{(j)}_{\ell}} f^{(j)} ( \xx^{(j)}_{\ell,i}) b^{(j)}_{\ell,i}, \quad f^{(j)} \in C(\Omega^{(j)}),
\end{align*}
with $ b^{(j)}_{\ell,i} $ as in \cref{thrm:NodalRepresentationMultilevelOperator}.

Inserting this into the combination technique representation of \cref{cor:CombinationTechniqueRepresentation} yields an improved representation of the \emph{tensor product multilevel operator}.

\begin{definition}\label{def:TPMLOperator}
In the setting described above, the \emph{tensor product multilevel operator} $ \cA^{\oomega}_T : C(\OOmega) \to \cV_T^{\oomega} $ is given as 
 \begin{align}\label{eq:TPMLOperator}
  \cA^{\oomega}_T(\ff)(\xx) = \sum_{\llambda \in \cJ^{\oomega}_T} 
  &\sum_{\substack{\bbeta \in \{0,1\}^d \\ \llambda + \bbeta \in \cI^{\oomega}_T}} (-1)^{\|\bbeta\|_1} \cdot  \nonumber \\
  &\cdot \sum_{i_1 = 1}^{N^{(1)}_{\lambda_1}} \cdots \sum_{i_d = 1}^{N^{(d)}_{\lambda_d}} f(\xx^{(1)}_{\lambda_1,i_1}, \dots, \xx^{(d)}_{\lambda_d,i_d}) \prod_{j=1}^d b^{(j)}_{\lambda_j,i_j}(\xx^{(j)}),
 \end{align}
 for every $ \ff \in C(\OOmega) $ and $ \xx = (\xx^{(1)}, \dots, \xx^{(d)}) \in \OOmega $.
\end{definition}

With the results of \cref{cor:PropertiesImageSpace}, we can obtain an estimate on the dimensionality of the sparse grid space $ \cV_T^{\oomega} $, that can be understood as the degrees of freedom of the approximation method.

\begin{theorem}\label{thrm:EstimateDimensionSparseGridSpace}
 In the setting described above, where we additionally assume that the sites $ X^{(j)}_i $ are quasi-uniform, the dimension the sparse tensor product space $ \cV_T^{\oomega} $ for the tensor product multilevel method can be estimated by 
 \begin{align*}
  \operatorname{dim}\left( \cV_T^{\oomega} \right) \leq C  \left(\max_{1 \leq j \leq d} \mu_j^{-\frac{n_j}{\omega_j}} \right)^T T^{R-1},
 \end{align*}
 where $ R $ counts how often the maximum is attained and $ C > 0 $ is a constant independent of $ \oomega $ and $ T $, but dependent on $ n_j $. 
\end{theorem}

\begin{proof}
 Since $ X^{(j)}_i $ are quasi-uniform, we can bound $ N^{(j)}_i \leq C \mu_j^{-i \cdot n_j} $. We prove the claim by induction over $ d $, following \cite[Proof of Theorem 4.1]{Griebel2013}. If $ d = 1 $, the statement is \cref{cor:PropertiesImageSpace}. 
 
 For the induction step $ d \leadsto d+1 $, we have the $ d+1$-directional sparse grid space 
 \begin{align*}
  \cV^{(\oomega, \omega_{d+1})}_T = \bigoplus_{\llambda^{\transpose} \oomega + \lambda_{d+1}\omega_{d+1} \leq T} \left(W^{(1)}_{\lambda_1} \otimes \cdots \otimes W^{(d)}_{\lambda_d} \otimes W^{(d+1)}_{\lambda_{d+1}} \right).
 \end{align*}
 Without loss of generality, we assume that 
 \begin{align}\label{eq:ProofDimensionalityAssumption}
  \mu_1^{-\frac{n_1}{\omega_1}} = \cdots = \mu_R^{-\frac{n_R}{\omega_R}} > \mu_{R+1}^{-\frac{n_{R+1}}{\omega_{R+1}}} \geq \mu_{R+2}^{-\frac{n_{R+2}}{\omega_{R+2}}} \geq \cdots \geq \mu_d^{-\frac{n_d}{\omega_d}} \geq \mu_{d+1}^{-\frac{n_{d+1}}{\omega_{d+1}}}.
 \end{align}
 By multilinearity of the tensor product we can write
 \begin{align*}
 \cV^{(\oomega, \omega_{d+1})}_T &= \bigoplus_{\lambda = 0}^{\lambda_{d+1,max}} \left( \bigoplus_{\llambda^{\transpose}\oomega \leq T - \omega_{d+1}\lambda} W^{(1)}_{\lambda_1} \otimes \cdots \otimes W^{(d)}_{\lambda_d} \right) \otimes W^{(d+1)}_{\lambda} \\
  &= \bigoplus_{\lambda = 0}^{\lambda_{d+1,max}} \cV^{\oomega}_{T - \omega_{d+1} \lambda} \otimes W^{(d+1)}_{\lambda}.
 \end{align*}
 The induction hypothesis implies that 
 \begin{align*}
   \operatorname{dim}(\cV^{\oomega}_{T - \omega_{d+1} \lambda}) \leq C \left(\max_{1 \leq j \leq d} \mu_j^{- \frac{n_j}{\omega_j}} \right)^{T - \omega_{d+1} \lambda} (T - \omega_{d+1} \lambda)^{R-1}.
 \end{align*}
 This, together with \cref{cor:PropertiesImageSpace} and \eqref{eq:ProofDimensionalityAssumption}, we obtain
 \begin{align*}
  &\operatorname{dim}(\cV^{(\oomega, \omega_{d+1})}_T) \leq \\
  &\leq C \sum_{\lambda = 0}^{\lambda_{d+1,max}}\left(\max_{1 \leq j \leq d} \mu_j^{- \frac{n_j}{\omega_j}} \right)^{T - \omega_{d+1} \lambda} (T - \omega_{d+1} \lambda)^{R-1} \mu_{d+1}^{-n_{d+1} \lambda} \\
  &\leq C \left(\max_{1\leq j \leq d} \mu_j^{-\frac{n_j}{\omega_j}}\right)^T T^{R-1} \sum_{\lambda = 0}^{\lambda_{d+1,max}} \left(\min_{1 \leq j \leq d} \mu_j^{ \frac{n_j}{\omega_j} \omega_{d+1}} \right)^{\lambda} \mu_{d+1}^{- n_{d+1} \lambda}\\
  &= C \left(\mu_1^{-\frac{n_1}{\omega_1}}\right)^T T^{R-1} \sum_{\lambda = 0}^{\lambda_{d+1,max}} \left(\min_{1 \leq j \leq d} \mu_j^{\frac{n_j}{\omega_j}} \right)^{\lambda \omega_{d+1}} \left(\mu_{d+1}^{-\frac{n_{d+1}}{\omega_{d+1}}} \right)^{\omega_{d+1} \lambda} \\
  &= C \left(\mu_1^{-\frac{n_1}{\omega_1}}\right)^T T^{R-1} \sum_{\lambda = 0}^{\lambda_{d+1,max}} \left(\min_{1 \leq j \leq d} \mu_j^{\frac{n_j}{\omega_j}} \mu_{d+1}^{-\frac{n_{d+1}}{\omega_{d+1}}} \right)^{\omega_{d+1} \lambda}.
 \end{align*}
 In view of \eqref{eq:ProofDimensionalityAssumption} we distinguish two cases: either 
 \begin{align*}
  \max_{1 \leq j \leq d} \mu_j^{-\frac{n_j}{\omega_j}} > \mu_{d+1}^{-\frac{n_{d+1}}{\omega_{d+1}}},
 \end{align*}
 which happens if $ R < d $ and may happen if $ R = d $ or
 \begin{align*}
  \max_{1 \leq j \leq d} \mu_j^{-\frac{n_j}{\omega_j}} = \mu_{d+1}^{-\frac{n_{d+1}}{\omega_{d+1}}},
 \end{align*}
 which can only happen if $ R = d $. 
 
 In the first case, we have
 \begin{align*}
  \left( \min_{1 \leq j \leq d} \mu_j^{\frac{n_j}{\omega_j}} \right) \cdot \mu_{d+1}^{-\frac{n_{d+1}}{\omega_{d+1}}} < 1,
 \end{align*}
 which means that 
\begin{align}\label{eq:ProofDimensionalityResult1}
 \operatorname{dim} \left(\cV^{(\oomega, \omega_{d+1})}_T \right) \leq C \left(\mu_1^{-\frac{n_1}{\omega_1}}\right)^T T^{R-1}. 
\end{align}
In the second case, we have 
 \begin{align*}
  \left(\min_{1 \leq j \leq d} \mu_j^{\frac{n_j}{\omega_j}} \right) \cdot \mu_{d+1}^{-\frac{n_{d+1}}{\omega_{d+1}}} = 1,
 \end{align*}
 which leads to 
 \begin{align}\label{eq:ProofDimensionalityResult2}
   \operatorname{dim} \left(\cV^{(\oomega, \omega_{d+1})}_T \right) \leq C \left(\mu_1^{-\frac{n_1}{\omega_1}}\right)^T T^{R-1} \sum_{\lambda = 0}^{\lambda_{d+1,max}} 1 \leq C \left(\mu_1^{-\frac{n_1}{\omega_1}}\right)^T T^R, 
 \end{align}
 using $ \lambda_{d+1,max} $ from \eqref{eq:Lambda_j,max} and the normalization $ \min \omega_j = 1 $. Combining \eqref{eq:ProofDimensionalityResult1} and \eqref{eq:ProofDimensionalityResult2} completes the step $ d \leadsto d+1 $ and the proof. 
\end{proof}

For special choices of $ \mu_j $ we recover \cite[Theorem 4.1]{Griebel2013}, up to a constant.

\begin{corollary}
 With the notations and assumptions of \cref{thrm:EstimateDimensionSparseGridSpace}, assume additionally that $ \mu_1 = \cdots = \mu_d = 1/2 $. Then the dimension of $ \cV^{\oomega}_T $ bounded by $ C 2^{T \max\{n_1/\omega_1, \dots, n_d/\omega_d\}}T^{R-1} $, where $ R $ counts how often the maximum is attained.
\end{corollary}

\subsection{Convergence Results}\label{subsec:ErrorEstimates}

In this section, we refine the error estimates of \cite{Kempf2023}. We allow different refinements of the involved point sets in each direction and provide error estimates in other norms than $ L_2 $.

\begin{theorem}\label{thrm:ErrorEstimateNew}
Assume the setting described at the beginning of \Cref{subsec:RepresentationTPML}. Let $ q \in [1, \infty] $ and $ q_1 = \dots = q_d = q $. For $ 1 \leq j \leq d $ let $ \tau_j $ be such that the error result of \cref{cor:ConvergenceMultilevel} holds with $ \alpha_j = \mu_j^{\sigma_j - \tau_j - \varepsilon_j - n_j (1/2 - 1/q)_+} $. Denote by $ \boldsymbol{\mathcal{F}}(\OOmega) $ either $ W^{\ttau,q}_{mix}(\OOmega) $, if $ q < \infty $, or $ C^{\ttau}_{b,mix}(\OOmega) $, if $ q = \infty $ and $ \ttau \in \N^d_0 $. Let $ \ff \in H^{\ssigma}_{mix}(\OOmega) $.

 Then there is a constant $ C > 0 $ such that
 \begin{align*}
  \| \left(\iiota - \cA_T^{\oomega} \right) \ff \|_{\boldsymbol{\mathcal{F}}(\OOmega)} \leq C \left(\max_{1 \leq j \leq d} \alpha_j^{\frac{1}{\omega_j}}\right)^T T^{\frac{P-1}{2}} \| \ff \|_{H^{\ssigma}_{mix}(\OOmega)},
 \end{align*}
 holds, where $ P $ counts how often the maximum is attained.
\end{theorem}

\begin{proof}
 We closely follow the ideas of \cite[Proof of Theorem 4.3]{Griebel2013} and prove the claim by induction over $ d $. For $ d = 1 $, the claim is simply the error result of \cref{cor:ConvergenceMultilevel}. 
 
 For the induction step $ d \leadsto d+1 $, without loss of generality, we assume that 
 \begin{align}\label{eq:ProofConvergenceAssumptionAlphas}
  \alpha_1^{\frac{1}{\omega_1}} = \cdots = \alpha_P^{\frac{1}{\omega_P}} > \alpha_{P+1}^{\frac{1}{\omega_{P+1}}} \geq \alpha_{P+2}^{\frac{1}{\omega_{P+2}}} \geq \cdots \geq \alpha_d^{\frac{1}{\omega_d}} \geq \alpha_{d+1}^{\frac{1}{\omega_{d+1}}}
 \end{align}
 holds.

 We can rewrite the $d+1$-directional Smolyak operator $ \cA^{(\oomega, \omega_{d+1})}_T $, see \cite{Griebel2013}, as 
 \begin{align*}
  \cA^{(\oomega, \omega_{d+1})}_T = \sum_{\lambda = 0}^{\lambda_{d+1,max}} \cA^{\oomega}_{T - \omega_{d+1} \lambda} \otimes A^{(d+1)}_{\lambda}.
 \end{align*}
 To obtain the stated bound, we use the error representation, see again \cite{Griebel2013},
 \begin{align*}
  \iiota^{(d+1)} - \cA^{(\oomega, \omega_{d+1})}_T &= \sum_{\lambda=0}^{\lambda_{d+1,max}} \left( \iiota^{(d)} - \cA^{\oomega}_{T - \omega_{d+1} \lambda} \right) \otimes A^{(d+1)}_{\lambda} \\
  &\phantom{=}+ \iiota^{(d)} \otimes \sum_{\lambda = \lambda_{d+1,max} + 1}^{\infty} A^{(d+1)}_{\lambda}.
 \end{align*}
 This means that we can estimate for any $ \ff \in H^{(\ssigma,\sigma_{d+1})}_{mix}(\OOmega \times \Omega^{(d+1)}) $,
 \begin{align}\label{eq:ProofConvergence1}
  \phantom{-5em}&\left\| \left(\iiota^{(d+1)} - \cA^{(\oomega, \omega_{d+1})}_T \right) \ff \right\|_{\boldsymbol{\mathcal{F}}(\OOmega \times \Omega^{(d+1)})} \leq \nonumber\\
  & \leq \sum_{\lambda=0}^{\lambda_{d+1,max}} \left\| \left(\left( \iiota^{(d)} - \cA^{\oomega}_{T - \omega_{d+1} \lambda} \right) \otimes A^{(d+1)}_{\lambda} \right) \ff \right\|_{\boldsymbol{\mathcal{F}}(\OOmega \times \Omega^{(d+1)})} + \\
  &\phantom{\leq}+ \left\| \left(\iiota^{(d)} \otimes \sum_{\lambda = \lambda_{d+1,max} + 1}^{\infty} A^{(d+1)}_{\lambda} \right) \ff \right\|_{\boldsymbol{\mathcal{F}}(\OOmega \times \Omega^{(d+1)})}. \nonumber
 \end{align}
 The induction hypothesis and \eqref{eq:ProofConvergenceAssumptionAlphas} allow us to bound 
 \begin{align*}
  &\left\| \left( \left( \iiota^{(d)} - \cA^{\oomega}_{T - \omega_{d+1} \lambda} \right) \otimes \iiota^{(d+1)} \right) \ff \right\|_{\boldsymbol{\mathcal{F}}(\OOmega \times \Omega^{(d+1)})} \leq \\
  &\leq C \max_{1 \leq j \leq d} \alpha_j^{\frac{T - \omega_{d+1} \lambda}{\omega_j}} (T- \omega_{d+1} \lambda)^{\frac{P - 1}{2}} \| \ff \|_{H_{mix}^{(\ssigma, \sigma_{d+1})}(\OOmega \times \Omega^{(d+1)})} \\
  &\leq C \alpha_1^{\frac{T - \omega_{d+1} \lambda}{\omega_1}} T^{\frac{P - 1}{2}} \| \ff \|_{H_{mix}^{(\ssigma, \sigma_{d+1})}(\OOmega \times \Omega^{(d+1)})}.
 \end{align*} 
 This, together with the error result in \cref{cor:ConvergenceMultilevel}, and since all norms are compatible, we can estimate the first term on the right-hand side of \eqref{eq:ProofConvergence1} by
 \begin{align*}
 & \left\| \left(\left( \iiota^{(d)} - \cA^{\oomega}_{T - \omega_{d+1} \lambda} \right) \otimes A^{(d+1)}_{\lambda} \right) \ff \right\|_{\boldsymbol{\mathcal{F}}(\OOmega \times \Omega^{(d+1)})} \leq \\
 &\leq C \alpha_1^{\frac{T - \omega_{d+1} \lambda}{\omega_1}} T^{\frac{P - 1}{2}} \alpha_{d+1}^{\lambda} \| \ff \|_{H_{mix}^{(\ssigma, \sigma_{d+1})}(\OOmega \times \Omega^{(d+1)})}. 
 \end{align*}
 Inserting this back into \eqref{eq:ProofConvergence1} yields
 \begin{align*}
   &\left\| \left(\iiota^{(d+1)} - \cA^{(\oomega, \omega_{d+1})}_T \right) \ff \right\|_{\boldsymbol{\mathcal{F}}(\OOmega \times \Omega^{(d+1)})} \leq \\
   &\leq C \left( \alpha_1^{\frac{T}{\omega_1}} T^{\frac{P-1}{2}} \sum_{\lambda = 0}^{\lambda_{d+1,max}} \alpha_1^{-\lambda \frac{\omega_{d+1}}{\omega_1}} \alpha_{d+1}^{\lambda} +  
+ \sum_{\lambda = \lambda_{d+1,max} + 1}^{\infty} \alpha_{d+1}^{\lambda} \right) \cdot \\
&\hspace{20em} \cdot \| \ff \|_{H_{mix}^{(\ssigma, \sigma_{d+1})}(\OOmega \times \Omega^{(d+1)})} \\
   &\leq C \left( \alpha_1^{\frac{T}{\omega_1}} T^{\frac{P-1}{2}} \sum_{\lambda = 0}^{\lambda_{d+1,max}} \left( \alpha_1^{-\lambda \frac{\omega_{d+1}}{\omega_1}} \alpha_{d+1}^{\lambda} \right) + 
    \alpha_{d+1}^{\frac{T}{\omega_{d+1}}} \right) \cdot\\
    &\hspace{20em} \cdot \| \ff \|_{H_{mix}^{(\ssigma, \sigma_{d+1})}(\OOmega \times \Omega^{(d+1)})} .
 \end{align*}
 Looking at the sum on the right hand side, we see that, if $ \alpha_1^{1/ \omega_1} > \alpha_{d+1}^{1/ \omega_{d+1}} $, the estimate $ \alpha_1^{-\lambda \frac{\omega_{d+1}}{\omega_1}} \alpha_{d+1}^{\lambda} < 1 $ follows. Hence,
 \begin{align*}
  &\left\| \left(\iiota^{(d+1)} - \cA^{(\oomega, \omega_{d+1})}_T \right) \ff \right\|_{\boldsymbol{\mathcal{F}}(\OOmega \times \Omega^{(d+1)})} \\
  &\leq C \left( \alpha_1^{\frac{T}{\omega_1}} T^{\frac{P-1}{2}} + \alpha_{d+1}^{\frac{T}{\omega_{d+1}}} \right) \| \ff \|_{H_{mix}^{(\ssigma, \sigma_{d+1})}(\OOmega \times \Omega^{(d+1)})} \\
  &\leq C \alpha_1^{\frac{T}{\omega_1}} T^{\frac{P-1}{2}} \| \ff \|_{H_{mix}^{(\ssigma, \sigma_{d+1})}(\OOmega \times \Omega^{(d+1)})}.
 \end{align*}
 On the other hand, if $ \alpha_1^{1/ \omega_1} = \alpha_{d+1}^{1/ \omega_{d+1}} $, which means that $ P = d $, we have 
 \begin{align*}
  &\left\| \left(\iiota^{(d+1)} - \cA^{(\oomega, \omega_{d+1})}_T \right) \ff \right\|_{\boldsymbol{\mathcal{F}}(\OOmega \times \Omega^{(d+1)})} \\
  &\leq C \left( \alpha_1^{\frac{T}{\omega_1}} T^{\frac{d-1}{2}}  \left( \frac{T}{\omega_{d+1}} + 1 \right) + \alpha_{d+1}^{\frac{T}{\omega_{d+1}}} \right) \| \ff \|_{H_{mix}^{(\ssigma, \sigma_{d+1})}(\OOmega \times \Omega^{(d+1)})} \\
  &\leq C \alpha_1^{\frac{T}{\omega_1}} T^{\frac{d+1}{2}} \| \ff \|_{H_{mix}^{(\ssigma, \sigma_{d+1})}(\OOmega \times \Omega^{(d+1)})}.
 \end{align*}
 This finishes the proof.
\end{proof}

The next two special cases of the error estimate in \cref{thrm:ErrorEstimateNew} show that  the result derived here recover the convergence rates of \cite{Griebel2013,Griebel2025} and also allows $ L_{\infty} $ error estimates.

\begin{corollary}
With the notation and assumptions of \cref{thrm:ErrorEstimateNew},
\begin{enumerate}
 \item in the case that $ q = 2 $ and $ \mu_j = \frac{1}{2} $, i.e., $ \alpha_j = 2^{- (\sigma_j - \tau_j - \varepsilon_j)} $, $ 1 \leq j \leq d $, we have the error bound 
 \begin{align*}
  \| \left(\iiota - \cA_T^{\oomega} \right) \ff \|_{H^{\ttau}(\OOmega)} \leq C 2^{-T\min\left\{ \frac{\sigma_1 - \tau_1 - \varepsilon_1}{\omega_1},\dots, \frac{\sigma_d - \tau_d - \varepsilon_d}{\omega_d} \right\}} T^{\frac{P-1}{2}} \| \ff \|_{H^{\ssigma}_{mix}(\OOmega)},
 \end{align*}
 where $ P $ counts how often the minimum is attained.
 \item in the case that $ q = \infty $, $ \tau_j = 0 $ and $ \mu_j = \frac{1}{2} $, i.e., $ \alpha_j = 2^{-\left(\sigma_j - \varepsilon_j - \frac{n_j}{2} \right)} $, $ 1 \leq j \leq d $, we have the error bound 
 \begin{align*}
  \| \left(\iiota - \cA_T^{\oomega} \right) \ff \|_{C_b(\OOmega)} \leq C 2^{-T\min\left\{ \frac{\sigma_1 - \varepsilon_1 - \frac{n_1}{2}}{\omega_1},\dots, \frac{\sigma_d - \varepsilon_d - \frac{n_d}{2}}{\omega_d} \right\}} T^{\frac{P-1}{2}} \| \ff \|_{H^{\ssigma}_{mix}(\OOmega)},
 \end{align*}
 where $ P $ counts how often the minimum is attained.
\end{enumerate}
\end{corollary}

Finally, again in analogy to \cite{Griebel2013,Griebel2025}, we can combine \cref{thrm:ErrorEstimateNew} and \cref{thrm:EstimateDimensionSparseGridSpace} and express the convergence rate in terms of the number of degrees of freedom, i.e., the dimension of the sparse grid space. This yields the cost complexity of the tensor product multilevel approximation of high-dimensional functions.

\begin{theorem}\label{cor:ErrorEstimateInN}
 With the notation and assumptions of \cref{thrm:ErrorEstimateNew}, denote by $ N := \operatorname{dim} \cV^{\oomega}_T $ the number of degrees of freedom and set 
 \begin{align}\label{eq:ErrorEstimateInNConvergenceOrder}
  \beta := \frac{\log\left( \max_{1\leq j \leq d}  \alpha_j^{\frac{1}{\omega_j}} \right)}{\log \left( \min_{1 \leq j \leq d} \mu_j^{\frac{n_j}{\omega_j}}  \right)}.
 \end{align}
 Assume that the maximum in the numerator is attained $ P $ times and the minimum in the denominator is attained $ R $ times. Then the approximation error of the tensor product multilevel method can be bounded by 
 \begin{align}\label{eq:ErrorEstimateInN}
  \| (\iiota - \cA_T^{\oomega}) \ff \|_{\boldsymbol{\cF}(\OOmega)} \leq C N^{-\beta} (\log N)^{\frac{P-1}{2} + \beta (R-1)} \| \ff \|_{H^{\ssigma}_{mix}(\OOmega)},
 \end{align}
 with a constant $ C > 0 $.
\end{theorem}

\begin{proof}
 In the case $ R = 1 $, \cref{thrm:EstimateDimensionSparseGridSpace} states that $ N \leq C \left(\max \mu_j^{-\frac{n_j}{\omega_j}} \right)^T $. Hence, we have 
 \begin{align*}
  N^{-\beta} &\geq C \left(\max_{1 \leq j \leq d} \mu_j^{-\frac{n_j}{\omega_j}} \right)^{- \beta T} = C \left(\min_{1 \leq j \leq d} \mu_j^{\frac{n_j}{\omega_j}} \right)^{\frac{\log\left( \max \alpha_j^{\frac{1}{\omega_j}} \right)}{\log \left( \min \mu_j^{\frac{n_j}{\omega_j}}  \right)} T} \\
  &\geq C \left( \max_{1\leq j \leq d}  \alpha_j^{\frac{1}{\omega_j}} \right)^T,
 \end{align*}
 where we used some algebraic manipulations using standard logarithmic identities.
 In view of \cref{thrm:ErrorEstimateNew} and $ T \leq \log(N) $, we obtain the stated estimate for $ R = 1 $.
 
 If $ R > 1 $, \cref{thrm:EstimateDimensionSparseGridSpace} states that $ N \leq C \left(\max \mu_j^{-\frac{n_j}{\omega_j}} \right)^T T^{R-1} $. Inserting this in
 \begin{align*}
  \left(\frac{N}{T^{R-1}} \right)^{-\beta}
 \end{align*}
 and following the same ideas as in the case $ R = 1 $ finishes the proof for $ R > 1 $.
\end{proof}

For special choices of $ q $, $ \mu_j $ and $ \tau_j $ we obtain an easier interpretable result.

\begin{corollary}
 \begin{enumerate}
  \item If $ q = 2 $ and $ \mu_j = \frac{1}{2} $, i.e., $ \alpha_j = 2^{- (\sigma_j - \tau_j - \varepsilon_j)} $, $ 1 \leq j \leq d $, the statement of \cref{cor:ErrorEstimateInN} holds with 
  \begin{align*}
   \beta = \frac{ \min\{ (\sigma_1 - \tau_1 - \varepsilon_1) / \omega_1, \dots, (\sigma_d - \tau_d - \varepsilon_d) / \omega_d \} }{\max\{ n_1 / \omega_1, \dots, n_d / \omega_d \}}.
  \end{align*}
 \item If $ q = \infty $, $ \tau_j = 0 $ and $ \mu_j = \frac{1}{2} $, i.e., $ \alpha_j = 2^{-\left(\sigma_j - \varepsilon_j - \frac{n_j}{2} \right)} $, $ 1 \leq j \leq d $, the statement of \cref{cor:ErrorEstimateInN} holds with 
 \begin{align*}
  \beta = \frac{ \min\{ (\sigma_1 - \varepsilon_1 - \frac{n_1}{2}) / \omega_1, \dots, (\sigma_d - \varepsilon_d - \frac{n_d}{2}) / \omega_d \} }{\max\{ n_1 / \omega_1, \dots, n_d / \omega_d \}}.
 \end{align*}
 \end{enumerate}
\end{corollary}

In accordance with \cite[Lemma 5.1]{Griebel2013} we can obtain the highest possible convergence rate $ \beta^* $ that can be achieved with the tensor product multilevel method.

\begin{lemma}\label{lem:BestRateBetaStar}
 With the notation and assumptions of \cref{thrm:ErrorEstimateNew} we can bound $ \beta $ from \eqref{eq:ErrorEstimateInNConvergenceOrder} by 
 \begin{align}\label{eq:BetaStar}
  \beta^* := \min_{1 \leq j \leq d} \frac{\sigma_j - \tau_j - \varepsilon_j - n_j \left( \frac{1}{2} - \frac{1}{q} \right)_+}{n_j}.
 \end{align}
\end{lemma}

\begin{proof}
 For fixed weight vector $ \oomega $ let $ k, m \in \{ 1, \dots, d \} $ be such that 
 \begin{align*}
  \mu_k^{\frac{1}{\omega_k}\left(\sigma_k - \tau_k - \varepsilon_k - n_k \left(\frac{1}{2} - \frac{1}{q} \right)_+ \right)} = \max_{1\leq j \leq d} \left( \mu_j^{\frac{1}{\omega_j}\left(\sigma_j - \tau_j - \varepsilon_j - n_j \left( \frac{1}{2} - \frac{1}{q} \right)_+ \right)} \right)  
 \end{align*}
 and 
 \begin{align*}
  \mu_m^{\frac{n_m}{\omega_m}} = \min_{1 \leq j \leq d} \mu_j^{\frac{n_j}{\omega_j}}.
 \end{align*}
 Since both terms are less than one this yields, for any $ 1 \leq j \leq d $,
 \begin{align*}
  \beta &= \frac{\log \left(\mu_k^{\frac{1}{\omega_k}\left(\sigma_k - \tau_k - \varepsilon_k - n_k \left( \frac{1}{2} - \frac{1}{q} \right)_+ \right)} \right)}{\log \left(\mu_m^{\frac{n_m}{\omega_m}} \right)} \\
  &\leq \frac{\log \left(\mu_j^{\frac{1}{\omega_j}\left(\sigma_j - \tau_j - \varepsilon_j - n_j \left( \frac{1}{2} - \frac{1}{q} \right)_+ \right)} \right)}{\log \left(\mu_j^{\frac{n_j}{\omega_j}} \right)} \\
  &= \frac{\sigma_j - \tau_j - \varepsilon_j - n_j \left( \frac{1}{2} - \frac{1}{q} \right)_+}{n_j},
 \end{align*}
 which leads to the stated $ \beta^* $.
\end{proof}

Finally, we can study the range of weights for which we obtain the highest possible convergence rate $ \beta^* $, similar to \cite[Lemma 5.2]{Griebel2013}.

\begin{theorem}\label{thrm:RangeOfOmegaForBestConvergenceRate}
Let $ 1 \leq k \leq d $ be the direction where 
\begin{align*}
\frac{\sigma_k - \tau_k - \varepsilon_k - n_k \left( \frac{1}{2} - \frac{1}{q} \right)_+}{n_k} = \beta^*.
\end{align*}
Then, for all weight vectors $ \oomega > 0 $ such that 
\begin{align}\label{eq:conditionWeights}
\frac{\sigma_k - \tau_k - \varepsilon_k - n_k \left( \frac{1}{2} - \frac{1}{q} \right)_+}{\sigma_i - \tau_i - \varepsilon_i - n_i \left( \frac{1}{2} - \frac{1}{q} \right)_+} \frac{\log(\mu_k)}{\log(\mu_i)} \leq \frac{\omega_k}{\omega_i} \leq \frac{n_k}{n_i} \frac{\log(\mu_k)}{\log(\mu_i)}, \quad 1 \leq i \leq d,
\end{align}
it holds that 
\begin{align*}
\beta = \frac{\log\left( \max_{1\leq j \leq d}  \alpha_j^{\frac{1}{\omega_j}} \right)}{\log \left( \min_{1 \leq j \leq d} \mu_j^{\frac{n_j}{\omega_j}}  \right)} = \beta^*.
\end{align*}
If \eqref{eq:conditionWeights} is not satisfied, then $ \beta < \beta^* $.
\end{theorem}

\begin{proof}
 The condition \eqref{eq:conditionWeights} yields on one side, since $ \mu_i < 1 $,
 \begin{align*}
  \frac{\sigma_k - \tau_k - \varepsilon_k - n_k \left( \frac{1}{2} - \frac{1}{q} \right)_+}{\omega_k}\log(\mu_k) \geq \frac{\sigma_i - \tau_i - \varepsilon_i - n_i \left( \frac{1}{2} - \frac{1}{q} \right)_+}{\omega_i} \log(\mu_i),
 \end{align*}
 which means that, using the definition of $ \alpha_i $ as in \cref{cor:ConvergenceMultilevel},
 \begin{align*}
  \log\left(\alpha_k^{\frac{1}{\omega_k}} \right) \geq \log \left(\alpha_i^{\frac{1}{\omega_i}} \right),
 \end{align*}
 and by monotony of the logarithm, keeping in mind that $ \alpha_j < 1 $,
 \begin{align*}
  \alpha_k^{\frac{1}{\omega_k}} \geq \alpha_i^{\frac{1}{\omega_i}}.
 \end{align*}
 Similarly, on the other side, we obtain
 \begin{align*}
   \mu_k^{\frac{n_k}{\omega_k}}\leq \mu_i^{\frac{n_i}{\omega_i}}.
 \end{align*}
 Together, we have 
 \begin{align*}
  \beta &= \frac{\log\left( \max_{1\leq j \leq d}  \alpha_j^{\frac{1}{\omega_j}} \right)}{\log \left( \min_{1 \leq j \leq d} \mu_j^{\frac{n_j}{\omega_j}}  \right)} 
  = \frac{\log\left(\alpha_k^{\frac{1}{\omega_k}} \right)}{\log \left( \mu_k^{\frac{n_k}{\omega_k}}  \right)} = \beta^*.
 \end{align*}
If \eqref{eq:conditionWeights} is not satisfied for all $ 1 \leq i \leq d $, it means that 
\begin{align*}
 \alpha_k^{\frac{1}{\omega_k}} > \alpha_i^{\frac{1}{\omega_i}} \quad \text{or} \quad \mu_k^{\frac{n_k}{\omega_k}} < \mu_i^{\frac{n_i}{\omega_i}}
\end{align*}
and therefore $ \beta \neq \beta^* $. In view of \cref{lem:BestRateBetaStar}, this means that $ \beta < \beta^* $.
\end{proof}

We obtain the following corollary for special choices of $ q $, $ \mu_i $ and $ \tau_i $.

\begin{corollary}
  \begin{enumerate}
  \item If $ q = 2 $ and $ \mu_j = \frac{1}{2} $, i.e., $ \alpha_j = 2^{- (\sigma_j - \tau_j - \varepsilon_j)} $, $ 1 \leq j \leq d $, condition \eqref{eq:conditionWeights} becomes
  \begin{align*}
  \frac{\sigma_k - \tau_k - \varepsilon_k}{\sigma_i - \tau_i - \varepsilon_i} \leq \frac{\omega_k}{\omega_i} \leq \frac{n_k}{n_i}
  \end{align*}
  and the statement of \cref{thrm:RangeOfOmegaForBestConvergenceRate} holds.
 \item If $ q = \infty $, $ \tau_j = 0 $ and $ \mu_j = \frac{1}{2} $, i.e., $ \alpha_j = 2^{-\left(\sigma_j - \varepsilon_j - \frac{n_j}{2} \right)} $, $ 1 \leq j \leq d $, condition \eqref{eq:conditionWeights} becomes
  \begin{align*}
  \frac{\sigma_k - \tau_k - \varepsilon_k - \frac{n_k}{2}}{\sigma_i - \tau_i - \varepsilon_i - \frac{n_i}{2}} \leq \frac{\omega_k}{\omega_i} \leq \frac{n_k}{n_i}
  \end{align*}
  and the statement of \cref{thrm:RangeOfOmegaForBestConvergenceRate} holds.
 \end{enumerate}
\end{corollary}

\subsection{Selection of Optimal Weights}\label{subsec:SelectionOfOptimalWeights}

With the results on the degrees of freedom in \cref{thrm:EstimateDimensionSparseGridSpace} and the convergence result of the low-dimensional approximations \cref{cor:ConvergenceMultilevel}, we are able to revisit the selection of optimal weights for the Smolyak method, cf. the discussion at the end of \Cref{subsec:SmolyakApproximation}. For ease of discussion, we assume that $ \lfloor T /\omega_j \rfloor = T /\omega_j $.

 \textbf{Equilibrating the accuracy.}
Equilibrating the accuracy in the extremal spaces $ V_{T/{\omega_j}}^{(j)} $, $ 1 \leq j \leq d $, yields the condition 
\begin{align*}
 \alpha_1^{\frac{T}{\omega_1}} = \alpha_2^{\frac{T}{\omega_2}} = \cdots = \alpha_d^{\frac{T}{\omega_d}},
\end{align*}
with $ \alpha_j = \mu_j^{\sigma_j - \tau_j - \varepsilon_j - n_j \left( \frac{1}{2} - \frac{1}{q} \right)_+} $. The choice 
\begin{align*}
 \omega_j = \left( \sigma_j - \tau_j - \varepsilon_j - n_j \left( \frac{1}{2} - \frac{1}{q} \right)_+ \right) \frac{\log \mu_j}{\log \mu},
\end{align*}
leads to an equilibrated convergence rate $ \mu^{T} $ with a freely choosable $ \mu \in (0,1) $. Then, we rescale $ \oomega $ such that $ \min \omega_j = 1 $ to satisfy the condition on the weight vector in \Cref{subsec:SmolyakApproximation}.

\textbf{Equilibrating the degrees of freedom.} Equilibrating the dimensionality of the extremal spaces $ V^{(j)}_{T/\omega_j}(C(\Omega^{(j)})) $, $ 1 \leq j \leq d $, yields the condition
\begin{align*}
 \mu_1^{-\frac{n_1}{\omega_1}T} = \mu_2^{-\frac{n_2}{\omega_2}T} = \cdots = \mu_d^{-\frac{n_d}{\omega_d}T} = \mu^{-T},
\end{align*}
where we omit the constant. This is satisfied if $ \omega_j = n_j \frac{\log \mu_j}{\log \mu} $ and leads to equilibrated degrees of freedom $ \mu^{-T} $ with a free $ \mu \in (0,1) $. Rescaling the resulting vector $\oomega $ to satisfy $ \min \omega_j = 1 $ then satisfies the condition on the weight vector in \Cref{subsec:SmolyakApproximation}.

\textbf{Equilibrating the cost-benefit rate.} Equilibrating the cost-benefit ratio leads to the condition 
\begin{align*}
\prod_{j=1}^{d} \mu_j^{- \frac{T}{\omega_j} \left( \sigma_j - \tau_j - \varepsilon_j - n_j \left( \frac{1}{2} - \frac{1}{q} \right)_+ + n_j \right)} = \mu^{-T},
\end{align*}
 with a freely choosable $ \mu \in (0,1) $.  This leads to 
 \begin{align*}
 \omega_j = d \left( \sigma_j - \tau_j - \varepsilon_j - n_j \left( \frac{1}{2} - \frac{1}{q} \right)_+ + n_j \right) \frac{\log \mu_j}{\log \mu}.
 \end{align*}
 Normalising $ \oomega $ to satisfy $ \min \omega_j = 1 $ yields the choice of the weight vector in this case.

\section{Implementation and Algorithm}\label{sec:Implementation}

The computation of the approximation to a high-dimensional function $ \ff \in C(\OOmega) $ using the operator of \cref{def:TPMLOperator} can be split into an offline and an online phase. We assume that the index set $ \cI_{T}^{\oomega} $ is already generated and the nested sequences $ (X^{(j)}_i) $, $ 1 \leq j \leq d $, $ 1 \leq i \leq \lambda_{j,max} $, are given or already generated, e.g., using the thinning algorithm proposed in \cite{Griebel2025}.

\subsection{The Offline Phase}

In the offline phase, we can precompute the functions $ b_{L,i} $ of \eqref{eq:NodalRepresentationMultilevelOperator}. Since this can be done for all directions $ 1 \leq j \leq d $ separately, we drop the direction dependence. The following discussion has essentially been done in \cite{Gollwitzer2026}. We repeat it here for convenience of the reader.

The functions $ b_{L,i} $ can be expressed as a linear combination of the scaled RBFs $ \Phi_{\ell} $, $ 1 \leq \ell \leq L $,
\begin{align*}
 b_{L,i} = \sum_{\ell=1}^{L} \sum_{k_{\ell} = 1}^{N_{\ell}} \alpha^{(i)}_{\ell,k_{\ell}} \Phi_{\ell}(\cdot - \xx_{\ell, k_{\ell}}),
\end{align*}
where the coefficients $ \aalpha^{(i)}_{\ell} $ are solutions of the uniquely solvable block linear system 
\begin{align}\label{eq:BlockLinearSystem}
 \begin{pmatrix}
      A_1 \\
    B_{21} & A_2    &        &            & \\
    B_{31} & B_{32} & A_3    &            & \\
    \vdots & \vdots & \cdots & \ddots     & \\
    B_{L1} & B_{L2} & \cdots & B_{L(L-1)} & A_L 
   \end{pmatrix}
   \begin{pmatrix}
    \aalpha^{(i)}_1 \\
    \aalpha^{(i)}_2 \\
    \aalpha^{(i)}_3 \\
    \vdots        \\
    \aalpha^{(i)}_L
   \end{pmatrix}
   = \begin{pmatrix}
      \chi_{L,i}|_{X_1} \\
      \chi_{L,i}|_{X_2} \\
      \chi_{L,i}|_{X_3} \\
      \vdots \\
      \chi_{L,i}|_{X_L}
 \end{pmatrix},
\end{align}
where $ A_{\ell} = ( \Phi_{\ell}(\xx_{\ell,i} - \xx_{\ell,k})) 
\in \R^{N_{\ell} \times N_{\ell}} $ and $ B_{\ell m} = 
(\Phi_m(\xx_{\ell,i} - \xx_{m,k} )) \in \R^{N_{\ell} 
\times N_m} $ and $ \chi_{L,i} \in W_L $ is the Lagrange function on level $ L $, i.e., the unique function in $ W_L $ that satisfies $ \chi_{L,i}(\xx_{L,k}) = \delta_{i,k} $ for $ \xx_{L,k} \in X_L $.

We have to solve the linear system \eqref{eq:BlockLinearSystem} for all $ 1 \leq i \leq N_L $. However, it turns out that this can be done efficiently. If $ i $ is such that $ \xx_{L,i} \in X_L \setminus X_{L-1} $, i.e., this point appears in the sequence $ (X_k) $ only in the last level, we see that $ \chi_{L,i}|_{X_{\ell}} = \00 \in \R^{N_{\ell}} $ for all $ 1 \leq \ell \leq L-1 $ and we only have to solve $ A_L \aalpha_L^{(i)} = \ee_i $ to compute the coefficients for this $ b_{L,i} $. If $ i $ is such that $ \xx_{L,i} \in X_{L-1} \setminus X_{L-2} $, we only have to solve the sub-system 
\begin{align*}
\begin{pmatrix}
A_{L-1}     & 0\\
B_{L(L-1)} & A_L 
\end{pmatrix}
\begin{pmatrix}
\aalpha^{(i)}_{L-1} \\
\aalpha^{(i)}_L
\end{pmatrix}
= \begin{pmatrix}
\chi_{L,i}|_{X_{L-1}} \\
\chi_{L,i}|_{X_L}
\end{pmatrix}.
\end{align*}
It holds again that $ \chi_{L,i}|_{X_L} = \ee_i $, and $ \chi_{L,i}|_{X_{L-1}} = \ee_k \in \R^{N_{L-1}} $, where $ k $ is the index of the position of $ \xx_{L,i} $ in $ X_{L-1} $. This can then be iterated for the remaining points. Only for those $ \xx_i $ with $ \xx_i \in X_1 $, we have to solve the whole system \eqref{eq:BlockLinearSystem}.

The computation of $ b_{L,i} $ is independent of the computation of $ b_{L,k} $ for $ i \neq k $. This means that the computation of all $ (b_{L,i}) $, $ 1 \leq L \leq \lambda_{max} $, $ 1 \leq i \leq N_L $, is highly parallelizable. The procedure is described in \cref{alg:TPMLOfflinePhase}.

\begin{algorithm}
 \caption{Offline Phase \label{alg:TPMLOfflinePhase}}

 \begin{algorithmic}[1]
  \Require Sequence of data sites $ (X_i)$
  \Ensure Coefficient vectors $ \aalpha^{(i)}_{\ell} $
  \For{$ L = 1, \dots, \lambda_{max} $}
		\For{$ i = 1, \dots, N_{L} $}
			\State Solve linear system \eqref{eq:BlockLinearSystem} for $ \aalpha^{(i)}_{\ell} $, $ 1 \leq \ell \leq L $ \;
		\EndFor
	\EndFor
	
 \end{algorithmic}
\end{algorithm}

With the arguments above it is easy to see the following statements, cf. \cite{Gollwitzer2026}.

\begin{proposition}\label{prop:ComputationalCost}
Assuming that each block of the linear system can be solved in linear time, computation of all necessary coefficients takes 
\begin{align*}
\cO \left(\sum_{\ell = 1}^{\lambda_{max}} N_{\ell} \log N_{\ell} \right) 
\end{align*}
time.

Storing all necessary coefficients in memory requires
\begin{align*}
\cO \left(\sum_{L=1}^{\lambda_{max}} \left(\sum_{\ell = 1}^{L} (N_{\ell} - N_{\ell-1}) \sum_{k = \ell}^{L} N_k \right) \right). 
\end{align*}
\end{proposition}

\subsection{The Online Phase}

With all coefficients computed, we only have to evaluate the functions $ b_{L,i} $ and combine the values as in \cref{def:TPMLOperator}. This can be done as described in \cref{alg:TPMLOnlinePhase}. Again, this is highly parallelizable.

\begin{algorithm}
 \caption{Online Phase \label{alg:TPMLOnlinePhase}}
 
 \begin{algorithmic}[1]
  \Require Target function $ \ff $, evaluation point $ \xx \in \OOmega $
  \Ensure Approximate value $ \cA_{T}^{\oomega} (\ff)(\xx) $
  \State Compute $ c_{\llambda} = \sum_{\substack{\bbeta \in \{0,1\}^d \\ \llambda + 		\bbeta \in \cI^{\oomega}_T}} (-1)^{\|\bbeta\|_1} $
  \For{$ j = 1, \dots, d $}
	\State Compute the values $ b_{\lambda_j,i_j}^{(j)} (\xx^{(j)}) $ for all $ i_j = 1, \dots, N_{\lambda_j}^{(j)} $
  \EndFor
 \end{algorithmic}
\end{algorithm}

It is also easy to obtain an estimate on the online costs of the method.

\begin{proposition}\label{prop:OnlineCost}
 A single evaluation of $ \cA^{\oomega}_T(\ff) $ takes 
 \begin{align*}
  \cO \left(\# \cJ^{\oomega}_T \prod_{j=1}^{d} N_{\lambda_j} \log N_{\lambda_j} \right)  
 \end{align*}
 time, assuming that evaluation of $ b_{\lambda_j, i_j} $ takes logarithmic time and $ c_{\llambda} $ can be computed in constant time.
\end{proposition}

\begin{remark}
\begin{enumerate}
 \item The assumptions of \cref{prop:ComputationalCost,prop:OnlineCost} are satisfied for compactly supported kernels $ \Phi $. If $ \Phi $ is non-local, we can use compression methods such as $ \cH $-matrix techniques, see, e.g., \cite{Boerm2010} or Samplets \cite{Harbrecht2022}, that may reduce the computational costs by not solving the block linear system \eqref{eq:BlockLinearSystem} $ N_L $ times but computing and storing the inverse of the system once.
 \item Further results about increasing the numerical efficiency can be found in \cite{Lot2026}.
\end{enumerate}
\end{remark}

\section{Numerical Experiments}\label{sec:NumericalExperiments}

We now present numerical experiments for the proposed method.

In high dimensions, the direct evaluation of global error norms is computationally prohibitive. In particular, accurately approximating an $ L_p(\OOmega) $-error requires prohibitively fine discretizations and high-order quadrature rules for $ p \neq \infty $. Hence, for all experiments, we only report $ \ell_{\infty} $-errors. For the experiments in \Cref{subsec:StandardSparseGrid} and \Cref{subsec:Genz}, we estimate the supremum error by random sampling. We choose a number of random samples in every direction, such that the full $d$-fold tensor product set consists of at least 125\,000 points, and compute the maximum pointwise error over the full tensor product. This procedure is repeated at least 4 times with independent samples, and the reported error is the maximum over all repetitions.
While this approach does not yield the true $L_{\infty} $-error, it provides a reliable empirical indicator of the approximation quality and is standard in high-dimensional numerical experiments. The reported errors in \Cref{subsec:Anisotropic} are computed on a custom set of error sites, described in more detail there.

\subsection{A Standard Sparse Grid Example}\label{subsec:StandardSparseGrid}

We investigate the computational cost of the TPML method for different
numbers of directions~$d$. To this end, we consider a fully isotropic
setting with weights $\oomega = \11$ and domains $\Omega^{(j)} = [0,1]$
for $1 \leq j \leq d$. Each domain is discretized by a uniform grid with
refinement parameter $\mu = 1/2$. The target function is chosen as
$\ff \equiv 1$ in order to eliminate the cost of function evaluations.
We fix the kernel $\Phi$ to be the $C^2$-Wendland kernel
\[
\phi_{1,1}(r) = (1-r)^3_+(3r+1),
\]
which is the reproducing kernel of $H^{2}(\R)$, and scale it such that
its support contains approximately ten grid points on each level.
All experiments were performed on a single thread of an Intel Xeon
E5-2680 CPU at 2.40\,GHz.

\cref{fig:CummulativeRuntime} shows the total runtime for a single point evaluation, the sum of
offline and online phases, as a function of the threshold $T$, for
$d = 3, 6, 9, 12$. The semilogarithmic scale reveals that the runtime grows roughly
exponentially in $T$ for all dimensions, consistent with the exponential
growth of $N = \dim \cV^{\oomega}_T$ in $T$ established in
\cref{thrm:EstimateDimensionSparseGridSpace}. The rate of growth
increases markedly with $d$: at fixed $T$, adding directions multiplies
the number of component grids, enlarging both the offline linear systems
and the online summation in \eqref{eq:TPMLOperator}.

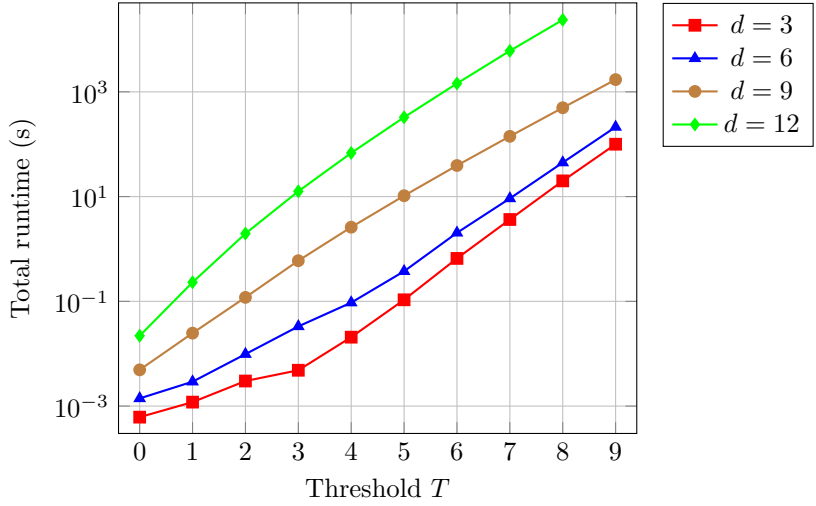
\begin{figure}
\begin{tikzpicture}
\begin{semilogyaxis}[
    xlabel = {Threshold $T$},
    ylabel = {Total runtime (s)},
    xmin = -0.4,  xmax = 9.4,
    ymin = 3e-4,  ymax = 5e4,
    xtick = {0,1,2,3,4,5,6,7,8,9},
    minor ytick = {},
    legend style = {at={(1.05,1)}, anchor=north west},
grid = both
]

\addplot[thick, color=red,  mark=square*,mark options=solid]
  coordinates {(0,6.15e-4)(1,1.19e-3)(2,3.01e-3)(3,4.83e-3)(4,2.07e-2)
               (5,1.07e-1)(6,6.60e-1)(7,3.654)(8,19.952)(9,100.068)};
\addlegendentry{$d = 3$}

\addplot[thick, color=blue,mark=triangle*,mark options=solid]
  coordinates {(0,1.40e-3)(1,2.92e-3)(2,9.81e-3)(3,3.32e-2)(4,9.45e-2)
               (5,3.75e-1)(6,2.034)(7,9.215)(8,44.615)(9,213.356)};
\addlegendentry{$d = 6$}

\addplot[thick, color=brown,mark=*,mark options=solid]
  coordinates {(0,4.92e-3)(1,2.47e-2)(2,1.19e-1)(3,5.97e-1)(4,2.605)
               (5,10.393)(6,39.253)(7,141.522)(8,497.103)(9,1714.365)};
\addlegendentry{$d = 9$}

\addplot[thick, color=green, mark=diamond*,mark options=solid]
  coordinates {(0,2.20e-2)(1,2.31e-1)(2,1.966)(3,12.610)(4,68.054)
               (5,326.958)(6,1445.58)(7,6062.42)(8,23723.7)};
\addlegendentry{$d = 12$}

\end{semilogyaxis}
\end{tikzpicture}
\caption{Total runtime for the standard sparse grid on $ [0,1]^d $ for $ d = 3,6,9,12 $ for a single point evaluation of $ \cA^{\11}_T(1) $ as a function of $ T $.}
\label{fig:CummulativeRuntime}
\end{figure}

\cref{tab:Runtimes} separates the total runtime into its offline
($t_{\mathrm{set}}$) and online ($t_{\mathrm{ev}}$) contributions.
For $d = 3$, the offline phase dominates at every threshold, the bold
entries always fall on $t_{\mathrm{set}}$, since the one-dimensional
linear systems grow with $T$ while a single evaluation of a 3-directional
combination formula remains cheap.
For $d = 6$, the two phases are nearly balanced throughout.
For $d = 9$ and $d = 12$, the online phase takes over from $T = 2$
onwards and eventually exceeds the offline phase by more than an order of
magnitude: at $T = 8$, $d = 12$, the evaluation costs nearly seventeen
times the setup.

\begin{sidewaystable}
\vspace{8cm}
\centering\small
\setlength{\tabcolsep}{5pt}
\caption{Runtimes (in seconds) for the offline ($t_{\mathrm{set}}$) and
online ($t_{\mathrm{ev}}$) phases of the TPML method on $[0,1]^d$
with isotropic weights $\oomega = \11$ and target function
$\ff \equiv 1$.
Bold entries indicate the dominant phase at each threshold.}
\label{tab:Runtimes}
\begin{tabular}{r
    rr
    rr
    rr
    rr
}
\toprule
& \multicolumn{2}{c}{$d = 3$}
& \multicolumn{2}{c}{$d = 6$}
& \multicolumn{2}{c}{$d = 9$}
& \multicolumn{2}{c}{$d = 12$} \\
\cmidrule(lr){2-3}\cmidrule(lr){4-5}\cmidrule(lr){6-7}\cmidrule(lr){8-9}
{$T$}
& {$t_{\mathrm{set}}$} & {$t_{\mathrm{ev}}$}
& {$t_{\mathrm{set}}$} & {$t_{\mathrm{ev}}$}
& {$t_{\mathrm{set}}$} & {$t_{\mathrm{ev}}$}
& {$t_{\mathrm{set}}$} & {$t_{\mathrm{ev}}$} \\
\midrule
0 & $\boldsymbol{6.4 \times 10^{-4}}$ & $5.1 \times 10^{-5}$
  & $\boldsymbol{1.2 \times 10^{-3}}$ & $2.0 \times 10^{-4}$
  & $\boldsymbol{4.0 \times 10^{-3}}$ & $9.4 \times 10^{-4}$
  & $\boldsymbol{1.4 \times 10^{-2}}$ & $7.8 \times 10^{-3}$ \\
1 & $\boldsymbol{1.0 \times 10^{-3}}$ & $1.7 \times 10^{-4}$
  & $\boldsymbol{2.1 \times 10^{-3}}$ & $8.4 \times 10^{-4}$
  & $\boldsymbol{1.5 \times 10^{-2}}$ & $1.0 \times 10^{-2}$
  & $\boldsymbol{0.126}$             & $0.105$ \\
2 & $\boldsymbol{2.6 \times 10^{-3}}$ & $4.5 \times 10^{-4}$
  & $\boldsymbol{5.9 \times 10^{-3}}$ & $4.0 \times 10^{-3}$
  & $0.056$ & $\boldsymbol{0.062}$
  & $0.866$ & $\boldsymbol{1.10}$ \\
3 & $\boldsymbol{4.2 \times 10^{-3}}$ & $6.1 \times 10^{-4}$
  & $\boldsymbol{0.019}$             & $0.015$
  & $0.222$ & $\boldsymbol{0.375}$
  & $4.31$  & $\boldsymbol{8.30}$ \\
4 & $\boldsymbol{0.019}$             & $1.6 \times 10^{-3}$
  & $\boldsymbol{0.047}$             & $0.047$
  & $0.733$ & $\boldsymbol{1.87}$
  & $17.3$  & $\boldsymbol{50.7}$ \\
5 & $\boldsymbol{0.104}$             & $2.8 \times 10^{-3}$
  & $\boldsymbol{0.211}$             & $0.164$
  & $2.21$  & $\boldsymbol{8.18}$
  & $59.0$  & $\boldsymbol{268}$  \\
6 & $\boldsymbol{0.653}$             & $7.0 \times 10^{-3}$
  & $\boldsymbol{1.52}$              & $0.518$
  & $6.79$  & $\boldsymbol{32.5}$
  & $180$   & $\boldsymbol{1266}$ \\
7 & $\boldsymbol{3.64}$              & $0.017$
  & $\boldsymbol{7.65}$              & $1.57$
  & $22.3$  & $\boldsymbol{119}$
  & $501$   & $\boldsymbol{5562}$ \\
8 & $\boldsymbol{19.9}$              & $0.039$
  & $\boldsymbol{40.1}$              & $4.56$
  & $83.3$  & $\boldsymbol{414}$
  & $1323$  & $\boldsymbol{22400}$ \\
9 & $\boldsymbol{100}$               & $0.091$
  & $\boldsymbol{201}$               & $12.8$
  & $346$   & $\boldsymbol{1368}$
  & \multicolumn{2}{c}{---} \\
\bottomrule
\end{tabular}
\end{sidewaystable}

\begin{figure}
\begin{tikzpicture} 
\begin{loglogaxis}[
xlabel={Number of SG points $N$}, 
ylabel={Relative $ \ell_{\infty} $-Error},
legend style={
        at={(1.05,1)},
        anchor=north west
    },
grid = both
]

\addplot[thick, color=red, mark = square*, mark options=solid] coordinates { 
(8,0.295268)
(20,0.0284344)
(50,0.0118325)
(123,0.0040213)
(297,0.001012)
(705,0.000327131)
(1649,6.84038e-05)
(3809,1.09155e-05)
(11777,2.37768e-06)
(28929,5.84017e-07)
};

\addplot[thick, color=orange, mark = 10-pointed star, mark options=solid] coordinates { 
(16,0.411956)
(48,0.0542485)
(136,0.0151987)
(368,0.00520534)
(961,0.00146655)
(2441,0.000277177)
(6065,8.1983e-05)
(14801,1.46806e-05)
(35585,4.44264e-06)
(84481,1.30049e-06)
};

\addplot[thick, color=violet, mark = star, mark options=solid] coordinates { 
(32,0.536573)
(112,0.0927749)
(352,0.0148748)
(1032,0.00511456)
(2882,0.000929153)
(7763,0.000259899)
(20333,6.28946e-05)
(52073,1.54299e-05)
(151393,3.05376e-06)
(405953,4.80686e-07)
};

\addplot[thick, color=blue, mark = triangle*, mark options=solid] coordinates { 
(64,0.669968)
(256,0.145506)
(880,0.0207403)
(2768,0.00808227)
(8204,0.00165456)
(23288,0.000343384)
(63953,0.000121919)
(171053,2.16062e-05)
(447713,5.39806e-06)
(1150753,1.16354e-06)
};

\pgfmathsetmacro{\Cval}{0.295268 / (8^(-1.5))}

\addplot[
    domain=8:50000,
    samples=200,
    thick,
    dashed,
    color=black
]
{ \Cval * x^(-1.5) };

\pgfmathsetmacro{\Cfour}{0.411956 / (16^(-1.5))}

\addplot[
    domain=16:120000,
    samples=200,
    thick,
    dashed,
    color=black
]
{ \Cfour * x^(-1.5) };

\pgfmathsetmacro{\Cfive}{0.536573 / (32^(-1.5))}

\addplot[
    domain=32:450000,
    samples=200,
    thick,
    dashed,
    color=black
]
{ \Cfive * x^(-1.5) };

\pgfmathsetmacro{\Csix}{900}

\addplot[
    domain=64:1500000,
    samples=200,
    thick,
    dashed,
    color=black
]
{ \Csix * x^(-1.5) };
 
\legend{$d=3$,$d=4$, $d=5$, $d=6$, $ N^{-1.5}$}
\end{loglogaxis}
\end{tikzpicture}
\caption{Relative errors for the isotropic TPML interpolation of $ \ff \equiv 1 $ for $ d = 3,4,5,6$. The observed decay matches the predicted rate of $ N^{-1.5} $.}
\label{fig:RelativeErrorsIsotropic}
\end{figure}
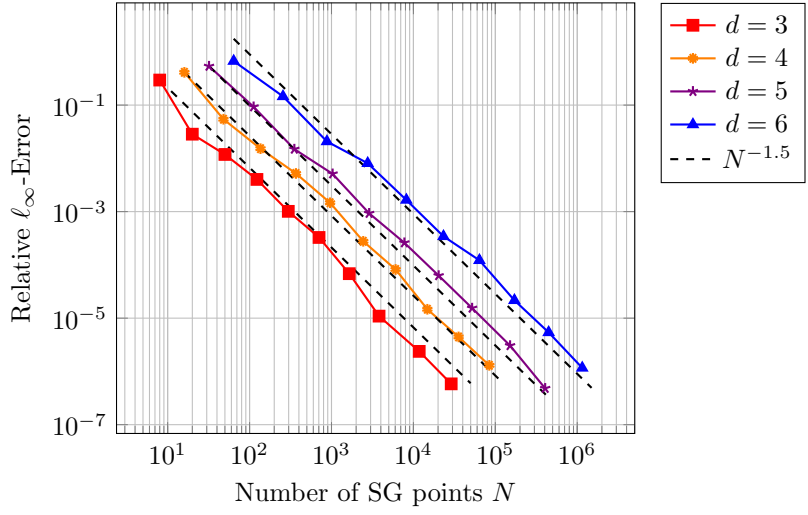

Additionally, we report the interpolation errors for the function $\ff \equiv 1$ in \cref{fig:RelativeErrorsIsotropic}, computed using the error estimator described at the beginning of this section and the same parameters as in the runtime experiments. 
The results are shown in a log-log plot for dimensions $d=3,4,5,6$. In all cases, the error exhibits a clear algebraic decay proportional to $N^{-1.5}$, in agreement with the theoretical prediction in \cref{cor:ErrorEstimateInN}, up to the polylogarithmic factor. Notably, this behaviour is already observed for small values of $N$, indicating that the asymptotic regime, in this easy example, is reached early.

\subsection{A $1+2+3$-dimensional Anisotropic Example}
\label{subsec:Anisotropic}

\begin{figure}
\centering
\begin{tikzpicture}
\begin{loglogaxis}[
    xlabel = {Number of SG points $N$}, ylabel = {Relative $\ell_\infty$-Error},
    legend style = {at = {(1.03,1.0)}, anchor = north west}, grid = both, width = 0.72\textwidth,
]
\addplot[thick, red, mark=square*, mark options=solid] coordinates {
  (1373508,0.0155606)
  (4208532,0.00520564)
  (12128017,0.00136535)
  (31673762,0.00109879)
  (81835352,0.000337737)
  (207322101,0.000239578)
};
\addplot[thick, blue, mark=triangle*, mark options=solid] coordinates {
  (771120,0.0217425)
  (2326488,0.01735)
  (6892708,0.00563152)
  (14705225,0.00501613)
  (33649894,0.00128264)
  (74946234,0.00109169)
  (153289908,0.00108256)
  (326919349,0.0003507684043)
};
\addplot[thick, brown, mark=*, mark options=solid] coordinates {
  (771120,0.0212064)
  (2326488,0.0171346)
  (6892708,0.00563152)
  (17793377,0.00135085)
  (40885162,0.00101108)
  (82624488,0.00100243)
  (193245833,0.000363453)};
\addplot[black, dashed, domain=6e5:3e8, samples=2] {2044*x^(-0.83333)};
\legend{Accuracy equ., DOF equ., Ratio equ., $N^{-5/6}$}
\end{loglogaxis}
\end{tikzpicture}
\caption{Relative $\ell_\infty$-error for the $1+2+3$ example as a function of the
number of sparse-grid points $N$ for the three weight strategies. The dashed guide
is the asymptotic rate $\beta^\ast = 5/6$. The per-step rate oscillates because the
anisotropic combination refines the directions in turn; the trend follows $\beta^\ast$.}
\label{fig:AnisotropicError}
\end{figure}
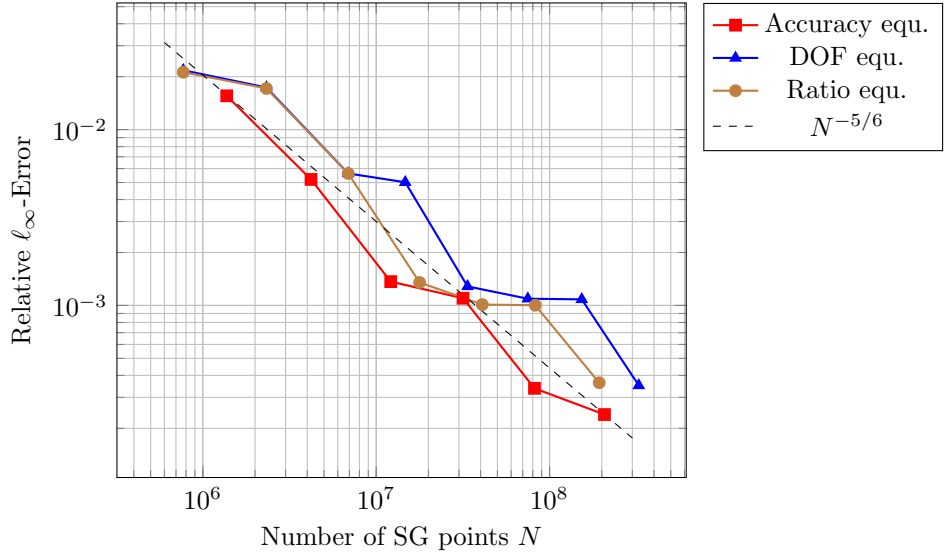

\begin{table}
\centering\small
\setlength{\tabcolsep}{4pt}
\caption{Raw point counts for the $1+2+3$ example. The left block lists the
univariate point-set sizes $N^{(j)}$ at refinement level $T$ (identical for all three
strategies); the right block lists the resulting sparse-grid size
$N=\dim V^{\oomega}_T$ at threshold $T$ for each weight strategy. Refining the two
$C^4$ directions most, the accuracy strategy is the most point-intensive, needing
roughly two to three times as many points as DOF at a given threshold.}
\label{tab:AnisotropicCounts}
\begin{tabular}{r rrr @{\hspace{1.6em}} rrr}
\toprule
 & \multicolumn{3}{c}{Univariate points $N^{(j)}$ (level $T$)}
 & \multicolumn{3}{c}{Sparse-grid points $N$ (threshold $T$)} \\
\cmidrule(lr){2-4}\cmidrule(lr){5-7}
$T$ & dir.\ 1 ($1$D) & dir.\ 2 ($2$D) & dir.\ 3 ($3$D) & accuracy & DOF & ratio \\
\midrule
1 & 9 & 116 & 420 & 1\,373\,508 & 771\,120 & 771\,120 \\
2 & 17 & 204 & 997 & 4\,208\,532 & 2\,326\,488 & 2\,326\,488 \\
3 & 33 & 353 & 2\,362 & 12\,128\,017 & 6\,892\,708 & 6\,892\,708 \\
4 & 65 & 637 & 5\,600 & 31\,673\,762 & 14\,705\,225 & 17\,793\,377 \\
5 & 129 & 1\,121 & 13\,273 & 81\,835\,352 & 33\,649\,894 & 40\,885\,162 \\
6 & 257 & 1\,977 & 31\,463 & 207\,322\,101 & 74\,946\,234 & 82\,624\,488 \\
7 & 513 & 3\,429 & 74\,578 & 511\,633\,111& 153\,289\,908 & 193\,245\,833 \\
\bottomrule
\end{tabular}
\end{table}

\begin{figure}
\centering
\begin{tikzpicture}
\begin{axis}[
x=3.0cm, 
height=4.0cm, 
xmin=-1,xmax=1, 
ymin=-0.05,ymax=1.05,
xlabel={$x$}, ylabel={$f^{(1)}(x)$}, title={$f^{(1)}$ on $\Omega^{(1)}=[-1,1]$},
    title style={font=\small}, tick label style={font=\footnotesize}, grid=both]
  \addplot[thick, black] table {target_dir0.dat};
  \addplot[red, only marks, mark=|, mark size=9pt] coordinates {(-0.5,0)(0.7,0)}; 
\end{axis}
\end{tikzpicture}\hfill
\begin{tikzpicture}
\begin{axis}[axis equal image, view={0}{90}, colormap/viridis reversed,
    point meta min=0, point meta max=1, enlargelimits=false,
    axis on top, xmin=-1,xmax=1, ymin=-1,ymax=1, xlabel={$x_1$}, ylabel={$x_2$},
    title={$f^{(2)}$ on $\Omega^{(2)}=B_1[\mathbf 0]$}, title style={font=\small},
    tick label style={font=\footnotesize}, colorbar, colorbar style={width=4pt}]
  \clip (axis cs:0,0) circle [radius=1];
  \addplot3[surf, shader=interp, mesh/ordering=x varies] table {target_dir1.dat};
  \draw[black,thick] (axis cs:0,0) circle [radius=1];
  \draw[black,thick] (axis cs:0.2,-0.3) circle [radius=0.6];
\end{axis}
\end{tikzpicture}

\vspace{0.6em}
\begin{tikzpicture}
\begin{axis}[axis equal image, view={0}{90}, colormap/viridis reversed,
    point meta min=0, point meta max=1, enlargelimits=false,
    axis on top, xmin=-1,xmax=1, ymin=-1,ymax=1, xlabel={$x_1$}, ylabel={$x_2$},
    title={$f^{(3)}$ on $\Omega^{(3)}=[-1,1]^3$, slice $x_3=c^{(3)}_3$},
    title style={font=\small}, tick label style={font=\footnotesize},
    colorbar, colorbar style={width=4pt}]
  \addplot3[surf, shader=interp, mesh/ordering=x varies] table {target_dir2.dat};
  \draw[black,thick] (axis cs:-1,-1) rectangle (axis cs:1,1);
  \draw[black,thick] (axis cs:0.2,-0.1) ellipse [x radius=0.6, y radius=0.7]; 
\end{axis}
\end{tikzpicture}
\caption{The three target factor functions $f^{(j)}(\xx)=(1-Q_j(\xx))_+^{\kappa_j}$,
$\kappa=(2,4,4)$, each with a conormal kink across the quadric $\{Q_j=1\}$: two
points in the $1$D interval, the circle in the $2$D disk, and (shown as the
central slice $x_3=c^{(3)}_3$) the ellipsoid in the $3$D cube. The target is the
product $\ff=f^{(1)}f^{(2)}f^{(3)}$ with $\|\ff\|_\infty=1$.}
\label{fig:AnisotropicTargets}
\end{figure}
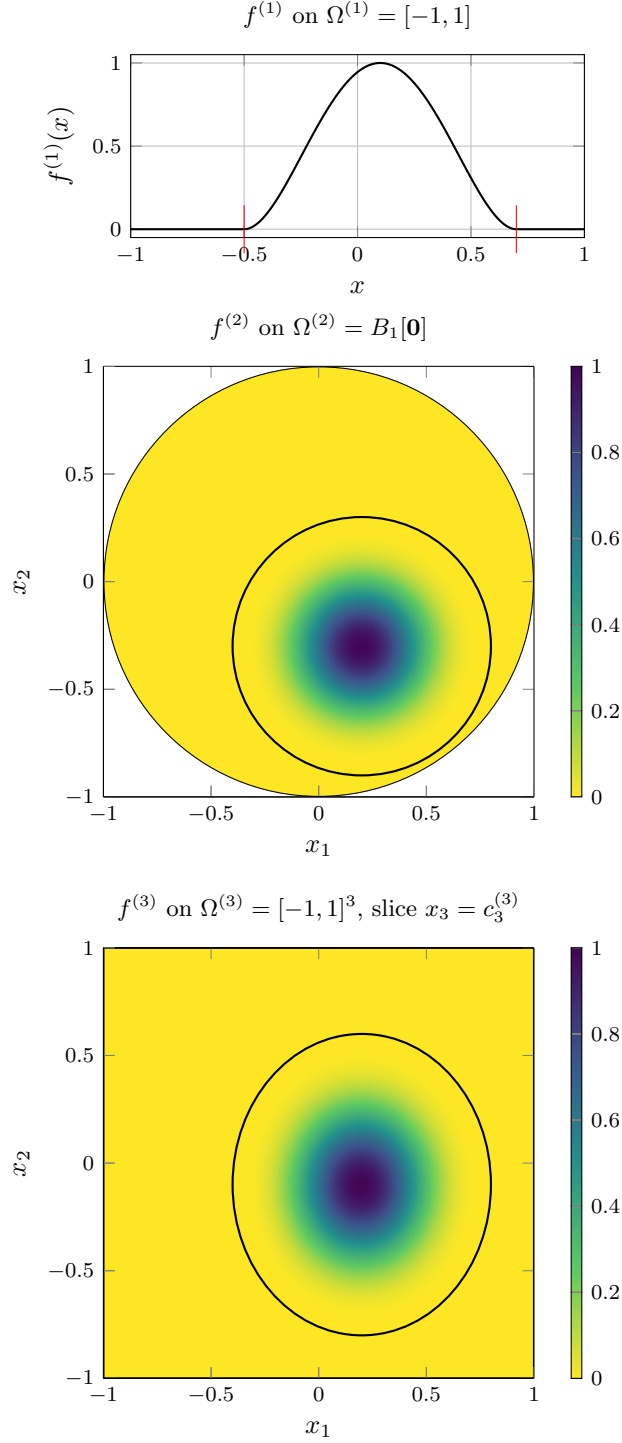

\begin{figure}
\centering
\begin{minipage}[b]{0.46\textwidth}\centering
\begin{tikzpicture}
\begin{axis}[width=\linewidth, height=3.2cm, xmin=-1.05,xmax=1.05, ymin=-0.5,ymax=0.5,
    ytick=\empty, 
	title={$\Omega^{(1)}$}, title style={font=\small},
    xtick={-1,0,1},
	axis y line=none, axis x line=middle,
    enlargelimits=false]
  \addplot[red, thick, only marks, mark=|, mark size=6pt] coordinates {(-0.5,0)(0.7,0)};
  \addplot[only marks, mark=*, mark size=1pt, blue] table[x=x, y expr=0] {errpts_dir0.dat};
\end{axis}
\end{tikzpicture}
\end{minipage}\hfill
\begin{minipage}[b]{0.46\textwidth}\centering
\begin{tikzpicture}
\begin{axis}[
width=\linewidth,
axis equal image, 
xlabel={$x_1$}, ylabel={$x_2$},
xtick={-1,0,1}, ytick={-1,0,1},
 tick label style={font=\footnotesize},
    xmin=-1.05,xmax=1.05, ymin=-1.05,ymax=1.05, title={$\Omega^{(2)}$}, title style={font=\small}]
  \draw[black!55,thick] (axis cs:0,0) circle [radius=1];
  \draw[red,thick] (axis cs:0.2,-0.3) circle [radius=0.6];
  \addplot[only marks, mark=*, mark size=1pt, blue] table {errpts_dir1.dat};
\end{axis}
\end{tikzpicture}
\end{minipage}\hfill

\vspace{0.8em}
\begin{minipage}[b]{0.62\textwidth}\centering
\begin{tikzpicture}
\begin{axis}[
width=1.35\linewidth, 
view={120}{22}, xmin=-1,xmax=1, ymin=-1,ymax=1, zmin=-1,zmax=1,
    xlabel={$x_1$}, ylabel={$x_2$}, zlabel={$x_3$}, 
    title={$\Omega^{(3)}$},
    title style={font=\small}, tick label style={font=\footnotesize},
    xtick={-1,0,1}, ytick={-1,0,1}, ztick={-1,0,1},
    scale mode=scale uniformly, grid=both]
  \addplot3[only marks, mark=*, mark size=1pt, blue] table {errpts_dir2.dat};
  \addplot3[ surf, opacity=0.28, fill opacity=0.22, faceted color=gray!50, shader=flat, domain=0:360, y domain=0:180, samples=24, samples y=16, z buffer=sort, ] ( {0.2 + 0.6*cos(x)*sin(y)}, {-0.1 + 0.7*sin(x)*sin(y)}, {0.15 + 0.5*cos(y)} );
\end{axis}
\end{tikzpicture}
\end{minipage}
\caption{The $31$/$106$/$108$ deterministic $\ell_\infty$ evaluation points (blue),
built once and reused at every threshold and for every weight strategy. Geometric
transverse shells straddle each kink quadric $\{Q_j = 1\}$ (the two red ticks in
$1$D, the red circle in $2$D, the ellipsoid in the $3$D cube), reaching in to
$\approx h_{\min}$, complemented by a domain-boundary and low-discrepancy interior
background. The $6$-dimensional test grid is the tensor product of the three sets.}
\label{fig:AnisotropicErrorPoints}
\end{figure}
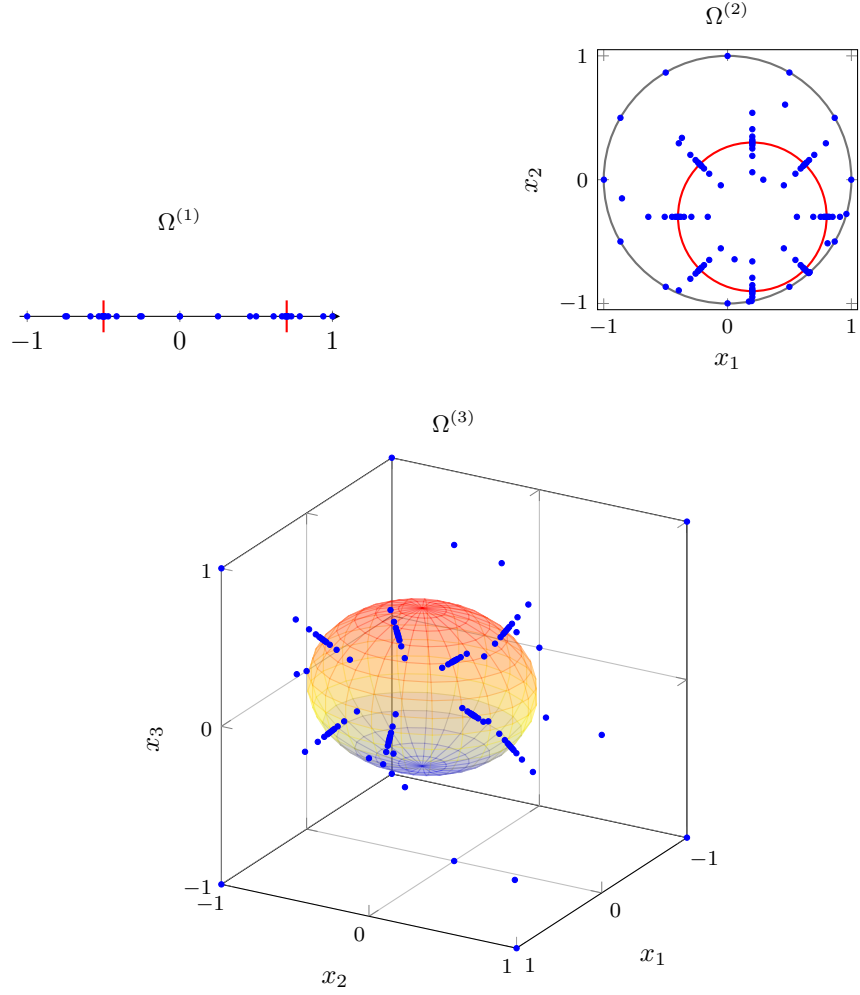

The second test exemplifies the versatility of the TPML method and tests the
convergence results and weight choices of
\Cref{sec:ConvergenceAnalysisAndOptimalWeights}. We use a $3$-directional setting
of total dimension $6$, but with deliberately different intrinsic dimensions,
domains and point generation methods in each
direction. For $j = 1$ we set $\Omega^{(1)} = [-1,1]$ and use uniform
grids with refinement parameter $\mu_1 = 1/2$. For $j = 2$ we set
$\Omega^{(2)} = B_1[\00]$, the closed unit ball in $\R^2$, and use the
thinning algorithm of \cite{Griebel2025} with $\mu_2 = 3/4$, starting from
$3 \times 10^6$ candidate points. For $j = 3$ we set $\Omega^{(3)} = [-1,1]^3$ and use Halton sequences with $\mu_3 = 3/4$.

The target function $\ff \colon \OOmega \to \R$ is a product of truncated-power
bumps, each with a conormal singularity across an interior quadric,
\begin{align*}
 \ff(\xx^{(1)}, \xx^{(2)}, \xx^{(3)}) = \prod_{j=1}^{3} f^{(j)}(\xx^{(j)}),
 \qquad
 f^{(j)}(\xx^{(j)}) = \left(1 - Q_j(\xx^{(j)})\right)_+^{\kappa_j},
\end{align*}
where $Q_j(\xx) = \sum_{i=1}^{n_j} \left((x_i - c^{(j)}_i)/a^{(j)}_i\right)^2$ and
$\kkappa = (2, 4, 4)$. The singular set $\{Q_j = 1\}$ is the pair of points
$c^{(1)}_1 \pm a^{(1)}_1$ in $\Omega^{(1)}$, an ellipse in $\Omega^{(2)}$, and an
ellipsoid in $\Omega^{(3)}$. The centres are $\cc^{(1)} = 0.1$,
$\cc^{(2)} = (0.2,-0.3)^\transpose$, $\cc^{(3)} = (0.2,-0.1,0.15)^\transpose$ and the
semi-axes are $\aa^{(1)} = 0.6$, $\aa^{(2)} = (0.6,0.6)$, $\aa^{(3)} = (0.6,0.7,0.5)$,
chosen so that each quadric lies strictly inside its domain, see
\cref{fig:AnisotropicTargets}. Since $(t)_+^{\kappa}$ lies in
$H^{\kappa + 1/2 - \varepsilon}$ and $Q_j$ has non-vanishing gradient on the regular
level set $\{Q_j = 1\}$, we have $f^{(j)} \in H^{\kappa_j + 1/2 - \varepsilon}(\Omega^{(j)})$
independently of the ambient dimension $n_j$, hence
$\ff \in H^{\ssigma}_{mix}(\OOmega)$ with
$\ssigma_{\mathrm{target}} = (2.5, 4.5, 4.5)$. The product is bounded with
$\|\ff\|_\infty = 1$, attained where all three $Q_j = 0$.

The kernels are chosen distinct per direction, matched in order to the intrinsic
dimension so that the three weight strategies differ in every direction. For
$j = 1$ we use the $C^2$-, and for $j = 2,3$ the $C^4$-Wendland kernel,
\begin{align*}
 \phi_{1,1}(r) = (1-r)_+^3(3r+1), \qquad
 \phi_{2,2}(r) = \phi_{3,2}(r) = (1-r)_+^6(35r^2 + 18r + 3),
\end{align*}
the reproducing kernels of $H^{2}(\R)$, $H^{3.5}(\R^2)$, and $H^{4}(\R^3)$,
respectively. We remark that the $C^4$-Wendland function has the same polynomial form in $\R^2$ and
$\R^3$ but a dimension-dependent native space. The overlap parameters are chosen per direction as
$\nnu = (12.0,\, 8.0,\, 4.0)$. This yields on average about
$10$, $50$, and $25$ neighbours in the scaled supports, keeping the local systems
well conditioned.

In each direction the first $o_j$ levels serve only to initialise the residual
hierarchy and are not included in the sparse grid combination. The approximation
quality at the first active level is of order $\alpha_j^{o_j}$ with
$\alpha_j = \mu_j^{\sigma_j - n_j/2}$. If $o_j$ were uniform across directions, the
slow-refining directions $j = 2,3$ ($\mu = 3/4$) would remain far from their
asymptotic regime while the fast-refining direction $j = 1$ ($\mu_1 = 1/2$) is
already well converged, causing non-monotone behaviour in the combination. We
therefore set $(o_1, o_2, o_3) = (3, 7, 7)$, giving the two slow directions the
longer warm-up so that all three multilevel operators enter their convergent regime
before the combination begins.

Because $\kappa_j + 1/2$ exceeds the order $\sigma_j$ of the native space in every
direction, the target lies strictly inside each native space, so the rate is
kernel-limited: the effective smoothness is $\ssigma = (2.0, 3.5, 4.0)$ and
\eqref{eq:BetaStar} gives
\begin{align*}
 \beta^* = \min_{j=1,2,3} \frac{\sigma_j - n_j/2}{n_j}
         = \min\!\left\{ \frac{2.0-0.5}{1},\, \frac{3.5-1}{2},\, \frac{4.0-1.5}{3} \right\}
         = \frac{5}{6},
\end{align*}
attained in direction $j = 3$. By \cref{thrm:RangeOfOmegaForBestConvergenceRate}, every
weight vector satisfying \eqref{eq:conditionWeights} achieves this rate. We compare
all three strategies of \Cref{subsec:SelectionOfOptimalWeights}. With
$a_j := \log_2(1/\mu_j) = (1,\, 0.415,\, 0.415)$ they evaluate to the following.

\textbf{Accuracy equilibrated weights.}
The accuracy formula $\omega_j \propto (\sigma_j - n_j/2)\,a_j$ gives, after
normalisation, $\oomega_{\mathrm{acc}} \approx (1.45,\, 1.00,\, 1.00)$. Directions
$2$ and $3$ receive the most levels: their slowly refining grids ($\mu = 3/4 $)
make the per-level error reduction $\alpha_{2,3} = \mu^{5/2}$ the largest among the
three, so they need the most refinement to balance the directional errors.

\textbf{Degree of freedom equilibrated weights.}
The DOF formula $\omega_j \propto n_j\,a_j$ gives
$\oomega_{\mathrm{DOF}} \approx (1.20,\, 1.00,\, 1.50)$. Direction $3$ has the
fastest-growing point sets and is therefore assigned the largest weight, capping its maximum level at
$\lfloor T/1.5 \rfloor$ and keeping the point counts roughly balanced.

\textbf{Cost-benefit equilibrated weights.}
The ratio formula $\omega_j \propto (\sigma_j + n_j/2)\,a_j$ gives
$\oomega_{\mathrm{ratio}} \approx (1.34,\, 1.00,\, 1.22)$. These interpolate between
the previous two: direction $3$ is penalised less than under DOF equilibration, as
its higher benefit-per-level partially offsets the rapid growth of its point sets.

\cref{tab:AnisotropicCounts} lists the univariate point counts per refinement level
together with the resulting sparse-grid sizes for the three strategies, showing how
the same direction-wise hierarchies are redistributed into sparse grids of very different
total size, with the accuracy strategy the most point-intensive.

The error is measured in a relative $\ell_\infty$ norm against the fixed constant
$\|\ff\|_\infty = 1$, on a deterministic set of evaluation points that is built once
and reused at every threshold and for every strategy. In each direction the points
consist of geometric shells straddling the kink quadric $\{Q_j = 1\}$ transversely,
down to the finest fill distance, together with a domain-boundary and
low-discrepancy interior background, see \cref{fig:AnisotropicErrorPoints}. The $6$-dimensional test grid is their tensor
product. As a validation, the location of
the maximum error is recorded at every threshold. It consistently sits at the kink of
direction $3$, where $\beta^*$ is attained, with $\ff$ nonzero there,
confirming that the reported error is the genuine finite-smoothness error of the
rate-limiting direction rather than a boundary or exterior-leakage artefact.

The three strategies converge monotonically over more than two orders of magnitude
in $N$, reaching a relative error of about $3 \times 10^{-4}$, as shown in
\cref{fig:AnisotropicError}. A least-squares fit of the error against $N$ gives
slopes of $0.84$ (accuracy), $0.77$ (ratio), and $0.69$ (DOF), clustering around the
predicted $\beta^* = 5/6 \approx 0.833$. The per-step rates oscillate because of the
anisotropic index set, in which the direction with the greatest gain in the error is not always the one that is refined, so every second threshold
grows $N$ without advancing the rate-limiting cube direction. The trend, captured by
the least-squares slope, follows $\beta^*$. Unlike the smooth benchmark of
\Cref{subsec:StandardSparseGrid}, whose asymptotic rate is visible already at small
$N$, the present target sits only half an order of smoothness inside each native
space, with the deficit concentrated on the kink quadrics, which enlarges the
convergence constant and delays the onset of the asymptotic regime. Nevertheless, the observed
slopes are consistent with $\beta^*$ up to this preasymptotic effect. The accuracy strategy reaches a given
error at a markedly lower point count than DOF and ratio, but uses for the same threshold $ T $ about double the number of points, see \cref{tab:AnisotropicCounts}.

\subsection{A $10$-dimensional Non-Tensorproduct Example}
\label{subsec:Genz}

To demonstrate the applicability of the proposed method beyond
tensor product functions, we test it on the \emph{oscillatory}
function of the Genz testing package \cite{Genz1984},
\begin{align*}
 \ff: [0,1]^{10} \to \R, \quad
 \ff(\xx) = \cos\!\left( 2\pi\theta + \sum_{j=1}^{10} c_j x_j \right),
\end{align*}
where $\theta = 0.5$ is a phase shift and $\cc \in \R^{10}$, $c_j > 0$,
controls the oscillation frequency. Following \cite{Novak1996}, we draw
$\widetilde{\cc}$ uniformly at random and normalize to
$\sum_{j=1}^{10} \widetilde{c}_j = 9.0$, then set $c_j = \widetilde{c}_j$.

Since $\ff \in C^{\infty}([0,1]^{10})$ has no a-priori identifiable
anisotropy, we set up a completely isotropic TPML with weight vector
$\oomega = \11$. As basis functions we use one-dimensional $C^4$-Wendland
kernels
\begin{align*}
 \phi_{1,2}(r) = (1-r)^5_+ (8r^2 + 5r +1),
\end{align*}
which are the reproducing kernels of $H^3(\R)$, with support radius chosen
so that on average $16$ points lie in the support. We discretize each
$[0,1]$ by a uniform grid with refinement parameter $\mu = 1/2$,
starting from $N_1 = 2$ points.

In \cref{tab:GenzRates} we report the number of sparse grid points $ N $, the relative $\ell_\infty$-error and the empirical algebraic convergence rates $ \beta_N $. The error decreases monotonically over roughly seven orders of magnitude up to $ T = 11 $, confirming that the method converges cleanly even though $ \ff $ is not of tensor product form. The pointwise rates $ \beta_N $ fluctuate between about $ 1.0 $ and $ 2.1 $ around an overall fitted value of $ \approx 1.4 $. 

The measured rate stays well below the asymptotic value of $ \beta^* = 5/2 $ of \eqref{eq:BetaStar}, and this is to be expected. For the isotropic setting the error estimate carries a polylogarithmic factor, so that at $ \beta = \beta^* = 5/2 $ and $ d = 10 $ the estimate reads 
\begin{align*}
 \| \ff - \cA_T^{\oomega}\ff \|_{L_{\infty}(\OOmega)} \leq C N^{-\frac{5}{2}} (\log(N))^{27},
\end{align*}
with exponent $ 4.5 + 9 \cdot 5/2 = 27 $. Its local slope is $ \beta^* - 27/(\log(N)) $, which for the grids reached here ($\log(N) \approx 16 $ to $ 21 $) lies between about $ 1.0 $ and $1.3 $, of the same order as the observed $ \beta_N $ once the level-to-level oscillation is averaged out. Isolating the bare algebraic rate would require the algebraic factor to dominate $ (\log(N))^{27} $. Pushing the effective slope even to $ 2 $ would need $ N \geq e^{54} \approx 3 \times 10^{23} $ points, far beyond any feasible computation. The experiment is therefore deep in the preasymptotic regime, and the observed rates are precisely those admitted by the estimate of \cref{cor:ErrorEstimateInN} rather than a contradiction of $ \beta^* $.

\begin{table}
\centering
\caption{Convergence rates for the $10$-dimensional Genz experiment. $ \beta_N $ denotes the empirical algebraic rate in $ N $.}
\label{tab:GenzRates}
\begin{tabular}{r r r r}
\toprule
$T$ & $N$ & Rel.\ error & $\beta_N$ \\
\midrule
1  & 6\,144           & $1.32 \times 10^{-1}$ & ---  \\
2  & 27\,904          & $1.71 \times 10^{-2}$ & 1.35 \\
3  & 109\,824         & $3.92 \times 10^{-3}$ & 1.08 \\
4  & 394\,624         & $1.06 \times 10^{-3}$ & 1.03 \\
5  & 1\,329\,408      & $1.38 \times 10^{-4}$ & 1.67 \\
6  & 4\,265\,248      & $2.58 \times 10^{-5}$ & 1.44 \\
7  & 13\,166\,688     & $2.45 \times 10^{-6}$ & 2.09 \\
8  & 39\,385\,108     & $8.02 \times 10^{-7}$ & 1.02 \\
9  & 114\,749\,288    & $1.90 \times 10^{-7}$ & 1.35 \\
10 & 326\,904\,337    & $1.98 \times 10^{-8}$ & 2.16 \\
11 & 913\,392\,437    & $6.83 \times 10^{-9}$ & 1.03 \\
\bottomrule
\end{tabular}

\end{table}

\begin{figure}
\centering
\begin{tikzpicture}
\begin{loglogaxis}[
    xlabel = {Number of SG points $N$},
    ylabel = {Relative $\ell_\infty$-Error},
	legend style= {
		at={(1.05,1.0)},
		anchor = north west
	},
    ymin = 1e-10, ymax = 5e-1,
    grid = both,
    minor tick num = 1
]

\addplot[thick, color=blue, mark=*, mark options=solid]
  coordinates {
    (6144,        1.32e-1)
    (27904,       1.71e-2)
    (109824,      3.92e-3)
    (394624,      1.06e-3)
    (1329408,     1.38e-4)
    (4265248,     2.58e-5)
    (13166688,    2.45e-6)
    (39385108,    8.02e-7)
    (114749288,   1.90e-7)
    (326904337,   1.98e-8)
    (913392437,   6.83e-9)
  };
\addlegendentry{Genz oscillatory}
\addplot[thick, dotted, domain=6144:913392437, samples=100]
    {2.65e4*x^(-1.4)};
\addlegendentry{$N^{-1.4}$}
\end{loglogaxis}
\end{tikzpicture}
\caption{Relative $\ell_\infty$-error for the TPML interpolation of the Genz
oscillatory function on $[0,1]^{10}$ with isotropic weights $\oomega = \11$
and $C^4$-Wendland kernels. The observed decay follows the
preasymptotic slope $N^{-1.4}$.}
\label{fig:GenzError}
\end{figure}
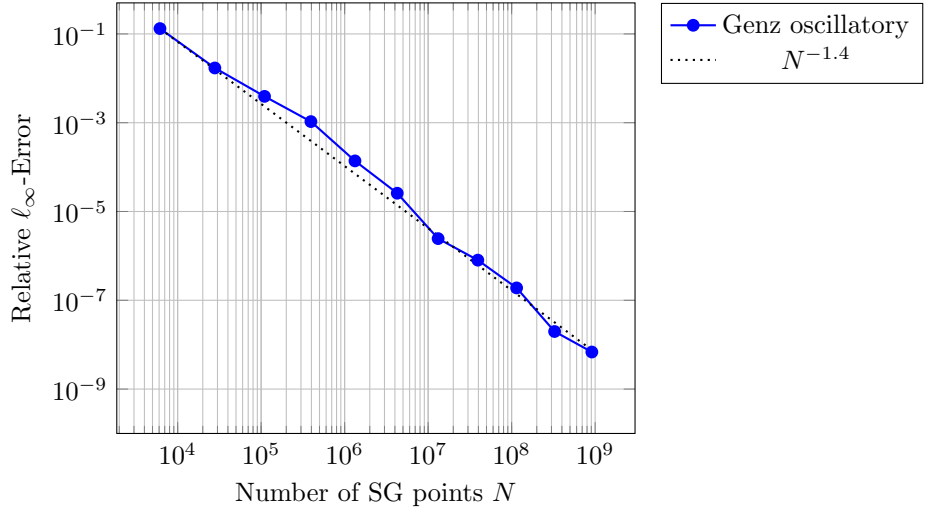

\section{Conclusion and Outlook}\label{sec:Conclusion}

We have revisited the tensor product multilevel method of \cite{Kempf2023}
and addressed three open problems left by that work.
Using the nodal representation of \cite{Gollwitzer2026}, we derived a
numerically feasible formulation of the high-dimensional TPML operator,
expressed as an explicit linear combination of data values weighted, in part, by
precomputed direction-wise basis functions.
Building on this, we established error estimates in general
$W^{\ttau,q}_{mix}$ and $C^{\ttau}_{b,mix}$ norms, extending the
existing $L_2$-theory.
Finally, we characterised the full range of weight vectors attaining
the optimal rate $\beta^*$ and derived three explicit weight strategies
that all achieve it.
The numerical experiments support the theoretical predictions and
demonstrate that the method is applicable beyond the tensor product
function class.

Several directions remain open. First, the cheap computation of the nodal representation currently requires nested point sets. Extending it to non-nested sequences
would broaden the method's applicability. Second, the weight strategies of
\Cref{subsec:SelectionOfOptimalWeights} rely on knowledge of the
smoothness parameters $\sigma_j$. A data-driven or adaptive selection
of $\oomega$ is desirable in practice. A dimension-adaptive variant,
in which the index set is built greedily from error indicators rather
than prescribed a priori, could handle functions with unknown
anisotropy. Finally, combining the offline phase with
$\mathcal{H}$-matrix compression \cite{Boerm2010} or Samplets
\cite{Harbrecht2022} would extend the method to non-local kernels
at nearly linear cost.


\newpage

\begin{thebibliography}{10}

\bibitem{Aubin2011}
{\sc J.~Aubin}, {\em Applied Functional Analysis}, Pure and Applied
  Mathematics: A Wiley Series of Texts, Monographs and Tracts, Wiley, 2011,
  \url{https://books.google.de/books?id=6iyejYSgM3wC}.

\bibitem{Avesani2025}
{\sc S.~Avesani, R.~Kempf, M.~Multerer, and H.~Wendland}, {\em Multiscale
  scattered data analysis in samplet coordinates}, SIAM Journal on Scientific
  Computing, 47 (2025), pp.~A3038--A3063,
  \url{https://doi.org/10.1137/24M1696305}.

\bibitem{Bartel2023}
{\sc F.~Bartel, M.~Sch{{\"a}}fer, and T.~Ullrich}, {\em Constructive
  subsampling of finite frames with applications in optimal function recovery},
  Applied and Computational Harmonic Analysis, 65 (2023), pp.~209--248,
  \url{https://doi.org/https://doi.org/10.1016/j.acha.2023.02.004},
  \url{https://www.sciencedirect.com/science/article/pii/S1063520323000192}.

\bibitem{Barthelmann2000}
{\sc V.~Barthelmann, E.~Novak, and K.~Ritter}, {\em High dimensional polynomial
  interpolation on sparse grids}, vol.~12, 2000, pp.~273--288.
\newblock Multivariate polynomial interpolation.

\bibitem{Bellman1957}
{\sc R.~E. Bellman}, {\em Dynamic programming}, Princeton University Press,
  Princeton, NJ, Oct. 1957.

\bibitem{Boerm2010}
{\sc S.~Borm}, {\em Efficient numerical methods for non-local operators}, EMS
  Tracts in Mathematics, European Mathematical Society, Z{\"u}rich,
  Switzerland, Dec. 2010.

\bibitem{Bungartz2004}
{\sc H.-J. Bungartz and M.~Griebel}, {\em Sparse grids}, Acta Numerica, 13
  (2004), p.~147–269, \url{https://doi.org/10.1017/S0962492904000182}.

\bibitem{Bungartz1994}
{\sc H.-J. Bungartz, M.~Griebel, D.~R\"oschke, and C.~Zenger}, {\em Two proofs
  of convergence for the combination technique for the efficient solution of
  sparse grid problems}, in Domain Decomposition Methods in Scientific and
  Engineering Computing, DDM7, D.~E. Keyes and J.~Xu, eds., Contemp. Math. 180,
  American Mathematical Society, Providence, 1994, pp.~15--20.

\bibitem{Buettner2025}
{\sc M.~B\"{u}ttner, R.~Kempf, and H.~Wendland}, {\em Numerical aspects of the
  tensor product multilevel method for high-dimensional, kernel-based
  reconstruction on sparse grids}, Journal of Scientific Computing, 106 (2025),
  \url{https://doi.org/10.1007/s10915-025-03144-0},
  \url{http://dx.doi.org/10.1007/s10915-025-03144-0}.

\bibitem{Conrad2013}
{\sc P.~R. Conrad and Y.~M. Marzouk}, {\em Adaptive smolyak pseudospectral
  approximations}, SIAM Journal on Scientific Computing, 35 (2013),
  pp.~A2643--A2670.

\bibitem{Floater1996}
{\sc M.~S. Floater and A.~Iske}, {\em Multistep scattered data interpolation
  using compactly supported radial basis functions}, Journal of Computational
  and Applied Mathematics, 73 (1996), pp.~65--78,
  \url{https://doi.org/https://doi.org/10.1016/0377-0427(96)00035-0},
  \url{https://www.sciencedirect.com/science/article/pii/0377042796000350}.

\bibitem{Garcke2007}
{\sc J.~Garcke}, {\em A dimension adaptive sparse grid combination technique
  for machine learning}, in Proceedings of the 13th Biennial Computational
  Techniques and Applications Conference, CTAC-2006, W.~Read, J.~W. Larson, and
  A.~J. Roberts, eds., vol.~48 of ANZIAM J., Dec. 2007, pp.~C725--C740.

\bibitem{Garcke2009}
{\sc J.~Garcke and M.~Hegland}, {\em Fitting multidimensional data using
  gradient penalties and the sparse grid combination technique}, Computing, 84
  (2009), pp.~1--25.

\bibitem{Genz1984}
{\sc A.~Genz}, {\em Testing multidimensional integration routines}, in Proc. of
  International Conference on Tools, Methods and Languages for Scientific and
  Engineering Computation, USA, 1984, Elsevier North-Holland, Inc., p.~81–94.

\bibitem{LeGia2010}
{\sc Q.~T.~L. Gia, I.~H. Sloan, and H.~Wendland}, {\em Multiscale analysis in
  sobolev spaces on the sphere}, {SIAM} J. Numer. Anal., 48 (2010),
  pp.~2065--2090.

\bibitem{Gollwitzer2026}
{\sc L.~Gollwitzer, R.~Kempf, and H.~Wendland}, {\em Nodal representations for
  kernel-based multilevel interpolation}, 2026,
  \url{https://arxiv.org/abs/2606.16643},
  \url{https://arxiv.org/abs/2606.16643}.

\bibitem{Griebel2012}
{\sc M.~Griebel and H.~Harbrecht}, {\em On the construction of sparse tensor
  product spaces}, Mathematics of Computation, 82 (2012), p.~975–994,
  \url{https://doi.org/10.1090/s0025-5718-2012-02638-x},
  \url{http://dx.doi.org/10.1090/S0025-5718-2012-02638-X}.

\bibitem{Griebel2013}
{\sc M.~Griebel and H.~Harbrecht}, {\em A note on the construction of l-fold
  sparse tensor product spaces}, Constructive Approximation, 38 (2013),
  p.~235–251, \url{https://doi.org/10.1007/s00365-012-9178-7},
  \url{http://dx.doi.org/10.1007/s00365-012-9178-7}.

\bibitem{griebel2025:2}
{\sc M.~Griebel and H.~Harbrecht}, {\em Kernel interpolation in sobolev spaces
  of hybrid regularity}, 2025, \url{https://arxiv.org/abs/2512.12684},
  \url{https://arxiv.org/abs/2512.12684}.

\bibitem{Griebel2025}
{\sc M.~Griebel, H.~Harbrecht, and M.~Multerer}, {\em Kernel interpolation on
  generalized sparse grids}, SIAM Journal on Mathematics of Data Science, 8
  (2026), pp.~335--361, \url{https://doi.org/10.1137/25M1761264},
  \url{https://doi.org/10.1137/25M1761264},
  \url{https://arxiv.org/abs/https://doi.org/10.1137/25M1761264}.

\bibitem{Griebel1990}
{\sc M.~Griebel, M.~Schneider, and C.~Zenger}, {\em A combination technique for
  the solution of sparse grid problems}, Forschungsberichte, TU Munich, TUM I
  9038 (1990), pp.~1--24,
  \url{https://api.semanticscholar.org/CorpusID:16460274}.

\bibitem{Hackbusch2012}
{\sc W.~Hackbusch}, {\em Tensor Spaces and Numerical Tensor Calculus},
  Spring\-er Series in Computational Mathematics, Springer Berlin Heidelberg,
  2012, \url{https://books.google.de/books?id=a5P71o6xcNMC}.

\bibitem{Hansen2010}
{\sc M.~Hansen}, {\em On tensor products of quasi-banach spaces}, tech. report,
  2010, \url{https://doi.org/10.3929/ETHZ-A-010388317},
  \url{http://hdl.handle.net/20.500.11850/154930}.

\bibitem{Harbrecht2022}
{\sc H.~Harbrecht and M.~Multerer}, {\em Samplets: Construction and scattered
  data compression}, J. Comput. Phys., 471 (2022), p.~111616.

\bibitem{Hegland2007}
{\sc M.~Hegland, J.~Garcke, and V.~Challis}, {\em The combination technique and
  some generalisations}, Linear Algebra and its Applications, 420 (2007),
  pp.~249--275,
  \url{https://doi.org/https://doi.org/10.1016/j.laa.2006.07.014},
  \url{https://www.sciencedirect.com/science/article/pii/S002437950600334X}.

\bibitem{Hegland2016}
{\sc M.~Hegland, B.~Harding, C.~Kowitz, D.~Pfl{\"{u}}ger, and P.~E. Strazdins},
  {\em Recent developments in the theory and application of the sparse grid
  combination technique}, in Software for Exascale Computing - {SPPEXA}
  2013-2015, H.~Bungartz, P.~Neumann, and W.~E. Nagel, eds., vol.~113 of
  Lecture Notes in Computational Science and Engineering, Springer, 2016,
  pp.~143--163, \url{https://doi.org/10.1007/978-3-319-40528-5\_7}.

\bibitem{Kempf2023}
{\sc R.~Kempf and H.~Wendland}, {\em High-dimensional approximation with
  kernel-based multilevel methods on sparse grids}, Numerische Mathematik, 154
  (2023), pp.~485--519,
  \url{http://nbn-resolving.org/urn:nbn:de:bvb:703-epub-7423-7}.

\bibitem{Light2006}
{\sc W.~Light and E.~Cheney}, {\em Approximation Theory in Tensor Product
  Spaces}, Lecture Notes in Mathematics, Springer Berlin Heidelberg, 2006,
  \url{https://books.google.de/books?id=vep7CwAAQBAJ}.

\bibitem{Lot2026}
{\sc F.~Lot and C.~Rieger}, {\em Efficiently parallelizable kernel-based
  multiscale algorithm}, IMA Journal of Numerical Analysis,  (2026),
  \url{https://doi.org/10.1093/imanum/draf127},
  \url{http://dx.doi.org/10.1093/imanum/draf127}.

\bibitem{Matern1986}
{\sc B.~Mat\'ern}, {\em Spatial Variation}, Springer New York, 1986,
  \url{https://doi.org/10.1007/978-1-4615-7892-5},
  \url{http://dx.doi.org/10.1007/978-1-4615-7892-5}.

\bibitem{Nobile2008}
{\sc F.~Nobile, R.~Tempone, and C.~G. Webster}, {\em An anisotropic sparse grid
  stochastic collocation method for partial differential equations with random
  input data}, SIAM Journal on Numerical Analysis, 46 (2008), pp.~2411--2442.

\bibitem{Novak1996}
{\sc E.~Novak and K.~Ritter}, {\em High-dimensional integration of smooth
  functions over cubes}, Numer. Math., 75 (1996), pp.~79--97.

\bibitem{Sickel2009}
{\sc W.~Sickel and T.~Ullrich}, {\em Tensor products of sobolev–besov spaces
  and applications to approximation from the hyperbolic cross}, Journal of
  Approximation Theory, 161 (2009), pp.~748--786,
  \url{https://doi.org/https://doi.org/10.1016/j.jat.2009.01.001},
  \url{https://www.sciencedirect.com/science/article/pii/S0021904509000197}.

\bibitem{Sickel2011}
{\sc W.~Sickel and T.~Ullrich}, {\em Spline interpolation on sparse grids},
  Appl. Anal., 90 (2011), pp.~337--383.

\bibitem{Smolyak1963}
{\sc S.~A. Smoljak}, {\em Quadrature and interpolation formulae on tensor
  products of certain function classes}, Dokl. Akad. Nauk SSSR, 148 (1963),
  pp.~1042--1045.

\bibitem{Ullrich2008}
{\sc T.~Ullrich}, {\em Smolyak's algorithm, sampling on sparse grids and
  {Sobolev} spaces of dominating mixed smoothness}, East J. Approx., 14 (2008),
  pp.~1--38.

\bibitem{Wendland1995}
{\sc H.~Wendland}, {\em Piecewise polynomial, positive definite and compactly
  supported radial functions of minimal degree}, Advances in Computational
  Mathematics, 4 (1995), p.~389–396,
  \url{https://doi.org/10.1007/bf02123482},
  \url{http://dx.doi.org/10.1007/BF02123482}.

\bibitem{Wendland2004}
{\sc H.~Wendland}, {\em Scattered Data Approximation}, Cambridge Monographs on
  Applied and Computational Mathematics, Cambridge University Press, 2004,
  \url{https://doi.org/10.1017/CBO9780511617539}.

\bibitem{Wendland2010}
{\sc H.~Wendland}, {\em Multiscale analysis in sobolev spaces on bound\-ed
  domains}, Numerische Mathematik, 116 (2010), pp.~493--517.

\bibitem{Wendland2017}
{\sc H.~Wendland}, {\em Multiscale radial basis functions}, in Frames and
  Oth\-er Bases in Abstract and Function Spaces : Novel Methods in Harmonic
  Analysis. Volume 1, Cham, 2017, Birkh{\"a}user, pp.~265--299.

\end{thebibliography}
\end{document}